\documentclass[11pt]{article}

\usepackage[english]{babel}
\usepackage{amsmath,amssymb, dsfont, amsthm}
\usepackage{epsfig}
\usepackage{bbm}
\usepackage{natbib}
\usepackage[colorlinks, linkcolor = black, citecolor = black, filecolor = black, urlcolor = blue]{hyperref} 
\usepackage{amsmath}
\usepackage{color}
\usepackage{multirow}
\usepackage{algorithm2e}
\usepackage{float}

   \usepackage{booktabs}
   \usepackage{multirow}
   \usepackage[table]{xcolor}
   \usepackage{wrapfig}
   \usepackage{float}
   \usepackage{colortbl}
   \usepackage{pdflscape}
   \usepackage{tabu}
   \usepackage{graphicx}
   \usepackage{caption}

   \graphicspath{{Pictures/}}

\definecolor{purple}{rgb}{0.55,0.2,0.90}

\newcommand{\PP}{{\mathbb P}}

\newcommand{\LL}{{\mathbb L}}

\newcommand{\RR}{\mathbb{R}}

\newcommand{\cF}{{\mathcal F}}

\newcommand{\Ex}{{\mathbb E}}

\newcommand{\NN}{\mathbb{N}}

\DeclareMathOperator{\cov}{cov}

\newcommand{\dd}{\mathrm d}

\theoremstyle{example}
\newtheorem{theorem}{Theorem}[section]
\newtheorem{lemma}[theorem]{Lemma}
\newtheorem{prop}[theorem]{Proposition}
\newtheorem{cor}[theorem]{Corollary}

\title{Testing for spherical and elliptical symmetry}
\author{Isaia Albisetti$^{1}$, Fadoua Balabdaoui$^{1}$ and  Hajo Holzmann$^2$\thanks{Corresponding author. Address: FB 12 - Mathematik und Informatik, Hans-Meerweinstr.~6, D-35043 Marburg, Germany. Email address:  holzmann@mathematik.uni-marburg.de }}

\begin{document}
\maketitle

\begin{center}
$^{1}$Seminar f\"ur Statistik, ETH Z\"urich, 8092, Z\"urich, Switzerland  \\
$^{2}$Fachbereich Mathematik und Informatik, Philipps-Universit\"at Marburg, D-35032 Marburg, Germany
\end{center}

\begin{abstract}	
We %review some of the tests for spherical symmetry and 
construct new testing procedures for spherical and elliptical symmetry based on the characterization that a random vector $X$ with finite mean has a spherical distribution if and only if $\Ex[u^\top X | v^\top X] = 0$ holds for any two perpendicular vectors $u$ and $v$. Our test is based on the Kolmogorov-Smirnov statistic, and its rejection region is found via the spherically symmetric bootstrap. We show the consistency of the spherically symmetric bootstrap test using a general Donsker theorem which is of some independent interest. 
For the case of testing for elliptical symmetry, the Kolmogorov-Smirnov statistic has an asymptotic drift term due to the estimated location and scale parameters. Therefore, an additional standardization is required in the bootstrap procedure. 
In a simulation study, the size and the power properties of our tests are assessed for several distributions and the performance is compared to that of several competing procedures.
\end{abstract}


\section{Introduction}
The distribution of a random vector $X \in \RR^d$ is said to be spherically symmetric if it stays invariant under orthogonal transformations, that is for any $d \times d$ matrix $\Gamma$ such that $\Gamma^{-1} = \Gamma^\top$ we have $\Gamma X \stackrel{\mathcal{L}}{=} X$, where $\stackrel{\mathcal{L}}{=}$ denotes equality in distribution. The random vector $X$ is called elliptically symmetric if it is spherically symmetric up to translation and rescaling. 
Spherically and elliptically symmetric distributions have drawn substantial interest as they present important and natural generalizations to the multivariate Gaussian distribution \citep{fang2018symmetric}. 
For example, in the simple linear model $Y = \mu + \epsilon$, where the expectation $\mu \in \RR^d $ is assumed to belong  to a $k$-dimensional linear subspace for some $k \le d$ and $\epsilon$ has a zero mean and a spherically symmetric distribution, the conclusion of the Gauss-Markov theorem on the optimality of the least squares estimator can be strengthened substantially \citep{berk1989optimality}. The classical t-test also extends to spherically symmetric distributions, see e.g.~\citet{cacoullos2014polar}.
%
%.  Let $\widehat{\mu}$ and  $\tilde{\mu}$ be the least squares estimator and a residual type estimator of $\mu$ respectively. If $\tilde{\mu}$ admits a finite covariance, then it follows from the Gauss-Markov theorem that the difference between the covariance of $\tilde{\mu}$ and $\widehat{\mu}$  is non-negative definite.  \cite{eaton86} shows that a sufficient condition for this property to hold is that $\Sigma^{-1/2} \epsilon$ is spherically symmetric, where $\Sigma$ is the covariance matrix of $\epsilon$ assumed to be positive definite. Note that this implies that the response $Y$ in the considered linear model has an elliptically symmetric distribution. 
%
In a single index regression model, where a response $Y \in \RR $ is linked to a given covariate $X \in \RR^d$ via the model $Y = \psi(\alpha^\top X)  + \epsilon$ where $\Ex[\epsilon | X]=0$, $\psi$ is some unknown ridge function and $\alpha$ is the unknown regression vector or index,
% usually assumed to the unit sphere to ensure identifiability. 
\citet{brillinger83} noted that if  the covariate $X$ has a non-degenerate Gaussian distribution, and  $\cov(\psi(\alpha^\top X), \alpha^\top X) \ne 0$ then the usual least squares estimator of $\alpha$ converges to a vector that is co-linear with $\alpha$, and hence a very simple estimator of $\alpha$ can be constructed. The same facts continue to hold when Gaussianity is replaced  by non-degenerate elliptical symmetry, see \citet[Lemma 2.28]{olive2014}. % 
% This semi-parametric model has received much attention for its flexibility over linear and generalized linear models. Estimating the index $\alpha$ without involving $\psi$ has been one approach followed by several  authors under specific regularity or shape assumptions on $\psi$; see e.g. \cite{han1987non}, \cite{hristache2001direct} and \cite{kakade2011}  to name only a few.  
% 
% 
%
The main feature 	is that if $X \in \RR^d$ admits an elliptically symmetric distribution with covariance matrix $\Sigma$, then for any $\beta \in \RR^d$ the conditional expectation
\begin{eqnarray}\label{PropEllipSym}
\Ex[X | \beta^\top X]  = \mu +  \frac{\beta^\top (X - \mu)}{\beta^\top \Sigma \beta} \Sigma \beta 
\end{eqnarray} 
is linear in $\beta^\top X$, a well-known property for Gaussian distributions, see \cite{cambanis81}.  
%
%In case $X$ has a non-degenerate covariance matrix $\Sigma$ and mean $\mu$, it follows from Corollary 5 of   that
%
\citet{DuanKerChau91} use \eqref{PropEllipSym} in their study of inverse regression. See also \citet[Section 5]{Baringhaus91} for an account on some interesting applications of testing spherical symmetry related to animal navigation, wind speed or paleomagnetic studies. 
Thus, testing for spherical or elliptical symmetry is an important problem, and various methods have been proposed in the literature. 
Building up on the work by \cite{smith77} for two-dimensional vectors, \cite{Baringhaus91} presents a family of tests for spherical symmetry in higher dimensions which exploit the fact that $X$ is spherically symmetric if and only if $\Vert X \Vert$ and $X/\Vert X \Vert$ are independent and that  $X / \Vert X \Vert $ is uniformly distributed on $\mathcal{S}_{d-1}$, the $d-1$-dimensional unit sphere. \cite{Baringhaus91} shows that the limiting distribution of the test statistics does not involve the unknown distribution of $X$.
%
%The test statistics $T_n$ involve a function $h$ defined on $[-1, 1]$ which has to satisfy some specific properties, among which the possibility of being extended into a series of Gegenbauer polynomials with nonnegative coefficients. The exact dependence of $T_n$ on $h$ and the observed data can be given under the form $T_n =  \int n^{-1}  \sum_{i, j} h(Z^T_i  Z_j) \mathbb{I}_{\max(U_i, U_j) \le u } dF_n(u)  $, where $Z_i = X_i/\Vert X_i \Vert$, $U_i = F(\Vert X_i \Vert)$ with $F$ the true distribution of $\Vert X_i \Vert$ and $F_n$ is the empirical distribution function of $U_1, \ldots, U_n$. \cite{Baringhaus91}  shows that, under $H_0$, $T_n$ converges weakly to  an infinite sum of weighted independent  random variables with chi-square distributions under the assumption that $F$ is continuous.  The arguments used to obtain the weak convergence rely on the theory of $V$-statistics. Consistency of the tests are also shown for any fixed alternative if  the sufficient condition that the coefficients in the expansion of $h$ in the Gegenbauer polynomials basis are all strictly positive holds true. One of the clear merits of the method is the fact that the limiting distribution of the tests statistics does not involve the unknown distribution of $X$. For example, \cite{Baringhaus91} shows that for the particular choice $h(t)  = (8t -2)/(17 - 8t), \ t \in [-1, 1]$ it turns out that the limit distribution of $T_n$ under $H_0$ is the same as that of the well-known Cram\'er-von Mises goodness-of-fit statistic.  
%
%
The test suggested by \cite{Fang93} uses the fact that $X$ has a spherically symmetric distribution if and only if the distribution of $u^\top X$ is the same  for any $u \in \mathcal{S}_{d-1}$. \cite{Fang93} then compares the distributions of $u^\top_k X$ for finitely many $u_k$ by using a two-sample Wilcoxon-type test. 
\cite{koltchi2000} suggest tests for spherical and elliptical symmetry by comparing the distribution of a random vector with its spherically symmetric projection using a Kolmogorov-Smirnov-type statistic. 

\cite{Liang2008} proposed a test 
using the fact that if $X$ is spherically symmetric such that $0$ is not an atom of its distribution, then $X/\Vert X \Vert$ is uniformly distributed on the sphere and hence a suitable transformation of $X/\Vert X \Vert$ to a $(d-1)$-dimensional vector involving the Beta distribution function has independent coordinates which are uniformly distributed on $[0,1]$.
%setting $Z_i = X_i/\Vert X_i \Vert$ and $Z_{ij}$ its $j$-th component the $(d-1)$-dimensional vectors
%\begin{eqnarray*}
%V_i = \left( F_1(Z_{i1}^2),  F_2((1-Z^2_{i1})^{-1} Z_{i2}^2), \ldots,  F_{d-1}\Big( \big(1-\sum_{j=1}^{d-2} Z^2_{ij}\big)^{-1}  Z_{i(d-1)}^2  \Big) \right)^\top 
%\end{eqnarray*}
%for  $i=1, \ldots, n$  are i.i.d.~uniformly distributed on $[0,1]^d$. 
%As the $d-1$ components $V_{ij}, j =1, \ldots, d-1$ of $V_i$ are also independent, one can either test uniformity of $V_{ij}, (i, j) \in \{1, \ldots, n \} \times \{1, \ldots, d-1\}$ using univariate uniform statistics such as the Watson or the Neyman's smooth test or multivariate uniform statistics as described in \cite[pp.~682-683]{Liang2008}. 
The method is easily implementable but does not yield a test which is universally consistent. 
%However, for random variables  $X_i$ with a  non-spherically symmetric distribution such that $X_i / \Vert X_i \Vert, i =1, \ldots, n$ are uniformly distributed on $\mathcal{S}_{d-1}$, these necessary tests will be unable to  reject  $H_0$.
%

\cite{Henze2014} proposed a test  based on the characterization that $X $ has a spherically symmetric distribution if and only if the characteristic function of $X$ takes the form $\phi_X(t)= \Ex[e^{i t^\top X}]  = g(\Vert t \Vert^2)$ for $t \in \RR^d$ and some function $g$. The test is based on the empirical characteristic function,  does not require moment assumptions, and the resulting test statistics are of Kolmogorov-Smirnov- as well as of Cram\'er-von Mises-type. The critical value is determined by using the spherically symmetric bootstrap, however without theoretical justification. 

Spherical and elliptical symmetry are also relevant notions in dynamic models as assumptions on the innovations. \citet{francq2017tests} extend the testing methodology from \cite{Henze2014} to testing for spherical symmetry of the innovation distribution in multivariate GARCH models. To derive the asymptotic distribution they use the central limit theorem for martingale difference sequences. In combination with a parametric bootstrap for the volatility matrix, they also extend the spherically symmetric bootstrap  to this setting.
Extensions of the methodology proposed in this paper to such testing problems might also be of some interest.

\subsection{Our contributions and organization of the manuscript}

In this paper, we propose tests for spherical and elliptical symmetry based on the characterization that a random vector $X \in \RR^d$, $d \ge 2$  with finite expectation has a spherically symmetric distribution if and only if
\begin{equation}\label{eq:eatonchara}
\Ex[u^\top X | v^\top X]=0 \quad \forall \ u,v \in \RR^d \quad \text{with } u^\top v = 0,
\end{equation}
see \citet[Theorem 1]{eaton86}. 
In Section \ref{sec: KStests} we consider spherical symmetry. For an appropriate modification of the characterization \eqref{eq:eatonchara} which does not require conditioning but only uses moment equations, we propose a consistent Kolmogorov-Smirnov-type asymptotic test in Section \ref{sec:characsphe}. Since it is not asymptotically distribution free, we develop a bootstrap version in Section \ref{sec:spherboot}. Indeed, in this section we give a consistency result for the spherically symmetric bootstrap for general VC-classes, which is of some independent interest and may be applied for other testing procedures. %such  as that suggested by \cite{Henze2014}.
Our proof does not use Poissonization but is rather based on a Donsker theorem for general processes \cite[Theorem 2.11.1]{vdvwell}. 
A discretized, implementable version of the Kolmogorov-Smirnov statistic is suggested in Section \ref{sec:discretealg}. 
In Section \ref{sec:extension} we propose extensions of our methodology to testing for elliptical symmetry. 
Finally, in Section \ref{sec: simss}, we investigate the level and power of the tests in finite samples for several distributions and compare the performance of our test to that of tests constructed by \cite{Baringhaus91}, \cite{Liang2008},  \cite{Henze2014} and \cite{koltchi2000}.  Our simulation results show that the new procedure performs very reasonably under both the null and alternative hypotheses.  
% We finish with some conclusions in Section \ref{sec: conc}. 
Proofs are deferred to an Appendix.

\section{The new test for spherical symmetry}\label{sec: KStests}

In this section, for a random vector $X \in \RR^d$ we consider testing
\begin{eqnarray*}
	H_0: X \ \textrm{is spherically symmetric} \  \quad  \textrm{versus} \ \quad H_1: X \ \textrm{is not.}
\end{eqnarray*}

Given $n \in \NN$ we let $X_1, \ldots, X_n$ be i.i.d.~random vectors with the same distribution $P$ as $X$. The empirical distribution of $X_1, \ldots, X_n$ is denoted by $\PP_n$. We let $\PP$ and $\Ex$ be the probability measure and the expected value operator on the probability space on which the $X_j$ are defined. 
We shall write $P f := \int f \dd P$ if the integral exists.

\subsection{Characterizing spherical symmetry}\label{sec:characsphe}

Our test will be based on the following adaptation of \citet[Theorem 1]{eaton86}.

\begin{lemma}\label{Theo.generate}
	Let  $X \in \RR^d$ be a random vector with finite expectation.  Then $X$ is spherically symmetric if and only if for any $c \in \RR$ and $u, v \in \mathcal{S}_{d-1}$ we have that 
	\begin{eqnarray*}
		\Ex[(v  - (v^\top u)  u)^\top X \mathds{1}_{u^\top X \ge c}] =0.
	\end{eqnarray*}
\end{lemma}

Define the family of functions on $\RR^d$ by
\begin{equation}\label{classF}
\mathcal{F}= \Big \{ f_{u, v, c}(x) = (v - (v^\top u) u)^\top x  \mathds{1}_{u^\top  x  \ge c} \mid   u, v \in \mathcal{S}_{d-1},  c  \in \RR \Big\}.
\end{equation}
By Lemma \ref{Theo.generate}, a random $d$-dimensional vector $X$ is not spherically symmetric (and hence the alternative hypothesis $H_1$ holds true) if and only if
\begin{eqnarray}\label{D}
%\sup_{ \substack{(u, v): S^2_{d-1}, u \perp v \\ \alpha \in \RR} } \Big \vert \Ex[v^\top X \mathbb{I}_{u^\top X \ge \alpha}  ]  \Big \vert  \ge \delta_0
\sup_{ f \in  \mathcal{F}} \big| P f  \big|  = \Delta_0(X) > 0.
\end{eqnarray}  
%
%\comf{do we need this additionnal notation $\delta_0(X)$? it can also be confused with the usual notation for a Dirac.}
%In the sequel,  $\Vert \cdot \Vert$ denotes the Euclidean norm of some random $X$.   Our starting point is a slight adaptation of .
%\bigskip
%
%
%
%\bigskip
%
%Corollary \ref{Theo.generate}  replaces the first characterization of spherical symmetry by an equivalent condition which does not involve anymore conditional expectations. Instead, the alternative condition invites to check whether the empirical version of the theoretical expectation $\Ex[(v  - v^\top u  u)^\top X \mathds{1}_{u^\top X \ge c }]  =0$ for all  $c \in \RR$  and $(u, v) \in \mathcal{S}^2_{d-1}$.  Assume that $\Ex[\Vert X \Vert^2] < \infty$.  Under $H_0$ and  for a fixed set of parameters  we have by the Central Limit Theorem that 
%\begin{eqnarray*}
%\sqrt n  \  \frac{1}{n} \sum_{j=1}^n (v  - v^\top u  u)^\top X_j  \ \mathds{1}_{u^\top  X_j  \ge c}  \ \to_d  \  \mathcal{N}\big(0, \sigma^2(u, v,c)\big) 
%\end{eqnarray*}
%with $\sigma^2(u, v, c)  = (v- v^\top u  u)^\top E\big[ X X^\top  \   \mathds{1}_{u^\top  X \ge c} \big] (v- v^\top u  u)$.  \\
%
\medskip
Consider the empirical process 
\begin{equation}\label{eq:empiricalproc}
\mathbb{G}_n  f = \sqrt n (\PP_n - P) f
\end{equation}
indexed by $f \in \mathcal{F}$.
By Lemma \ref{Theo.generate}, 
\begin{equation}\label{eq:charac}
\mathbb{G}_n  f = \sqrt n \, \PP_n \,f\ \text{ for all } \ f \in \mathcal{F} \ \text{ if and only if } \ H_0 \text{ holds true.}
\end{equation}
%denote the empirical process indexed by the elements of $\mathcal{F}$, with $\PP$ the true probability measure associated with the distribution generating .  Under spherical symmetry, Corollary \ref{Theo.generate} implies that $\PP f  =0$ for all $f \in \mathcal{F}$.  It can be shown that the pointwise weak convergence above can be strengthened to weak convergence of the whole process $\mathbb G_n$ . More formally, we have the following theorem. 
%
%\medskip
%
\begin{theorem}\label{Weakconv}
	Assume that  $X$ is spherically symmetric and satisfies $\Ex[\|X\|^2_2]< \infty$. Consider the space $\ell^\infty(\mathcal{F})$ of all uniformly bounded functions on $\mathcal{F}$ with respect to the supremum norm. Then there exists a tight centered Gaussian process  $\mathbb{G}$ with covariance function 
	\begin{equation}\label{propG}
	%E\big[(\mathbb G f_{u, v, M, \epsilon})^2\big]  &= &  \sigma^2(u, v, M, \epsilon)  \   \  \textrm{and} \notag \\
	\Ex\big[\mathbb G f_{u, v, c} \mathbb G f_{\tilde{u}, \tilde{v}, \tilde{c}}\big]  =   (v- (v^\top u) u)^\top \Ex \big[   X X^\top \ \mathds{1}_{u^\top X \ge c, \tilde{u}^\top X \ge \tilde{c}}\big] (\tilde{v}  - (\tilde{v}^\top \tilde{u}) \tilde{u}),
	\end{equation}
	such that the weak convergence 
	\begin{eqnarray*}
		\mathbb G_n \Rightarrow \mathbb G, \qquad n \to \infty, 
	\end{eqnarray*}
	in $\ell^\infty(\mathcal{F})$ holds true. 
\end{theorem}

The proof of Theorem \ref{Weakconv} is based on showing that $\mathcal{F}$ is a VC-subgraph class of functions, that is a collection of real-valued functions for which the subgraphs have a finite VC-index, see \citet[Section 2.6.2]{vdvwell}.

In the sequel, we adopt the usual notation $\Vert \nu_n \Vert_{\mathcal{D}}: = \sup_{f \in \mathcal{D}} \vert \nu_n(f) \vert$ for a stochastic process $\nu_n$ defined on a class of functions $\mathcal{D}$.  
Consider the statistic  
\begin{eqnarray}\label{Tn}
T_n :=  \sqrt{n} \Vert  \PP_n \Vert_{\mathcal{F}}.
\end{eqnarray} 
Given $\alpha \in (0,1)$ let $q_\alpha(\mathbb G)$ denote the $\alpha$-quantile of $\Vert  \mathbb G \Vert_{\mathcal{F}}$, which will depend on the distribution of $X$. 
%
%{\color{red} Absolute continuity of $\Vert  \mathbb G \Vert_{\mathcal{F}}$?}
%
An asymptotic  test based on $T_n$ rejects $H_0$ if $T_n > q_{1-\alpha}(\mathbb G)$. By \eqref{eq:charac} and Theorem \ref{Weakconv} such a test has asymptotic level $\alpha$. 
We show that the test is also consistent against fixed alternatives.   
%Note that $T_n =  \sqrt{n} \Vert  \PP_n \Vert_{\mathcal{D}}$ under the hypothesis of spherical symmetry. Also, an immediate consequence  of  Theorem \ref{Weakconv}   of the weak convergence is that 
%\begin{eqnarray*}
%T_n \to_d  \Vert \mathbb G\Vert_{\mathcal{D}}.
%\end{eqnarray*}
%Suppose that $\mathcal{D} = \mathcal{F}$ and let fix $\alpha \in (0,1)$. If  $q_\alpha$ be the $(1-\alpha)$-quantile of $\Vert \mathbb{G} \Vert_{\mathcal{F}}$, then the test which rejects $H_0$ if and only if $T_n  > q_{\alpha}$ has an asymptotic level equal $\alpha$.   To show consistency of this test  under a fixed alternatives,  suppose that    Then, we have the following result.
%
\begin{prop}\label{cons}
	Under the alternative hypothesis, that is if $X$ does not have a spherically symmetric distribution, we have that 
	%	Let $T_n$ be the statistic defined in (\ref{Tn}) with $\mathcal{D} = \mathcal{F}$.  If  $X$ satisfies  the condition in (\ref{D}), then
	\begin{eqnarray*}
		\lim_{n \to \infty}   \PP \Big(T_n  > q_{1-\alpha}(\mathbb G) \Big)  =1.   
	\end{eqnarray*}
\end{prop}

\subsection{The spherically symmetric bootstrap}\label{sec:spherboot}

Asymptotic quantiles of tests for spherical symmetry can be estimated by using the spherically symmetric bootstrap.
Indeed, if $\PP(X=0) = 0$, $X$ is spherically symmetric if and only if the random variable $X / \|X\|$ is uniformly distributed on $\mathcal{S}_{d-1}$, and is independent of $\|X\|$; see for example \cite{cambanis81}. Hence, we have the representation
\begin{eqnarray*}
	X  \stackrel{d}{=}   R\,  U , 
\end{eqnarray*}
where $U \sim \mathcal{U}(\mathcal{S}_{d-1})$ is uniformly distributed on $\mathcal{S}_{d-1}$ and $R \stackrel{d}{=} \|X\| $ is a positive random variable which is independent of $U$.   
Thus, in order to  provide new data under the null hypothesis of spherical symmetry,  we shall employ the spherically symmetric bootstrap as investigated in  \citet[Example 3]{romano1989bootstrap} and in \citet{koltchinskii1998testing}: We bootstrap the observed norms  $\Vert X_1 \Vert, \ldots, \Vert  X_n \Vert$ and multiply the obtained  values with vectors that are independently sampled from the unit sphere.  

More precisely, 
%let $L_0$ denote the distribution of $\Vert X_1 \Vert$ and 
write $L^*_n$  for the empirical distribution of $R_1=\Vert X_1 \Vert, \ldots, R_n=\Vert X_n \Vert$. In the bootstrap procedure, we draw independent samples $(R^*_1, \ldots, R^*_n)$ from $L^*_n$ and $(W_1, \ldots, W_n)$  from the uniform distribution $ \mathcal{U}(\mathcal{S}_{d-1})$ on the $(d-1)$-dimensional unit sphere. For a class of functions $f  \in \cF_s$ for which $P f = 0$ vanishes if $P$ is a spherically symmetric distribution, the spherically symmetric bootstrapped empirical  process is defined as 
\begin{eqnarray}\label{Gnstar}
\mathbb{G}^* _n = \sqrt n \mathbb P^*_n, \  \ \textrm{where} \ \  \mathbb P^*_n  := n^{-1}  \sum_{i=1}^n \delta_{R^*_i  W_i}. 
\end{eqnarray}
The following result extends and complements previous consistency results for the spherically symmetric bootstrap by \citet{romano1989bootstrap} and \citet{koltchinskii1998testing}.

%Below we give a proof of this proposition, we note that the bootstrap procedure used in this problem cannot be handled with standard consistency arguments for the usual sample with replacement bootstraps. Note also that our arguments are different from those used by\cite{koltchi2000}. Even though these authors also use empirical process theory, we provide a proof which relies on Theorem 2.11.1 of \cite{vdvwell} on Central Limit Theorems for stochastic processes.   Another motivation to use  that Theorem is that it offers well a well-structured result that can be used for similar or related problems.\\

\begin{theorem}\label{CondEmpProc}
	Suppose that $X$ is spherically symmetric and that $\PP(X=0)= 0$. Consider a VC-subgraph class of functions $\mathcal{F}_s$ with a square-integrable envelope function $F$, such that $P f = 0$ for all $f \in \cF_s$.
	
	Let $R^*_1, \ldots, R^*_n$ be an i.i.d. sample from the empirical distribution of $R_1 = \Vert X_1 \Vert, \ldots, R_n =\Vert  X_n \Vert$. Also, let $(W_1, \ldots, W_n)$ be a random sample from $\mathcal{U}(\mathcal{S}_{d-1})$  taken to be  independent of $(R^*_1, \ldots, R^*_n)$.  Then, for almost all $X_1, X_2, \ldots $  it holds that
	\begin{eqnarray*}
		\mathbb{G}^*_n  \Rightarrow \mathbb{G}
	\end{eqnarray*}
	in $\ell^\infty(\mathcal{F}_s)$, where $\mathbb{G}$ is a tight Gaussian process on $\ell^\infty(\mathcal{F}_s)$ with covariance function $P(f_1\, f_2)$, $f_i \in \cF_s$. 	
\end{theorem}

\subsection{Discretization}\label{sec:discretealg}

For an actual implementation, the test based on the Kolmogorov-Smirnov statistic in \eqref{Tn} needs to be discretized, that is, the supremum has to be computed over a finite number of elements from $\cF$ in \eqref{classF}. 
We shall choose the points on the sphere uniformly at random, while we use a discretization of a sufficiently large interval as values for the parameter $c$.
More precisely, for an integer $N_u  \ge 1$ sample $ U_1, \ldots, U_{N_u}, V_1, \ldots, V_{N_u}$ independently from the uniform distribution  on $\mathcal{S}_{d-1}$ such that  they are also independent of $X_1, \ldots, X_n$.  Write 
\[ \underline{U}_{N_u} = \{U_1, \ldots, U_{N_u}\}\]
and similarly for $\underline{V}_{N_u} $.
Let also $c_0 \ge 2$ and $N_c \in \NN$, and consider the (random) family of functions 
\begin{eqnarray}\label{DN}
\mathcal{F}_{N_u, N_c, c_0}  & :=&  \Big \{f_{u, v, c} \mid u \in \underline{U}_{N_u}, \ v \in \underline{V}_{N_u},    \\
&& \hspace{0.6cm} c \in \{-c_0  +  2c_0 (j/N_c) \mid j =0, \ldots, N_c \} \Big \}.\notag
\end{eqnarray}
The discretized version of our test statistic $T_n$ in \eqref{Tn} will be
\[ \widetilde{T}_{n, {N_u, N_c, c_0}} : = \sqrt n \Vert  \PP_n \Vert_{\mathcal{F}_{N_u, N_c, c_0}}.\]
We shall derive the asymptotic distribution and consistency when $N_u, N_c, c_0 \to \infty$ as $n \to \infty$. 

To this end, consider the full index set $\Theta =  \mathcal{S}_{d-1} \times \mathcal{S}_{d-1} \times \RR$
for the function class $\mathcal{F} = \mathcal{F}_\Theta$ in \eqref{classF}. Given $\theta \in \Theta$ we let $n(\theta)$ denote the element $(u, v, c) \in \underline{U}_{N_u} \times \underline{V}_{N_u} \times \{-c_0  +  2c_0 (j/N_c) \mid j =0, \ldots, N_c \}$ of minimal distance to $\theta$. 
%Further we let 
%
%\[ \mathcal{F}_{\Theta}^{(n)} = \{ f_{n(i)} \mid i \in \Theta\}.\]
%
\begin{theorem}\label{convdiscrete}
	Assume that $\Ex[\|X\|^2_2]< \infty$, that $X$ has a continuous distribution, and that $N_u, N_c, c_0 \to \infty$ such that $c_0 / N_c \to 0$ as $n \to \infty$. 
	%and consider the space $\ell^\infty(\Theta)$ of all uniformly bounded functions on $\Theta$ with respect to the supremum norm. 
	Then for the empirical process $\mathbb G_n$ in \eqref{eq:empiricalproc} we have that 
	\begin{eqnarray*}
		\{\mathbb G_n(f_{n(\theta)}) \mid \theta \in \Theta\} \Rightarrow \{\mathbb G(f_{\theta}) \mid \theta \in \Theta\}, \qquad n \to \infty 
	\end{eqnarray*}
	in $\ell^\infty(\Theta)$, where $\mathbb G$ is the Gaussian process from \eqref{propG}.
\end{theorem}
From this result we deduce consistency of the test based on the discretized test statistic. 
\begin{cor}\label{cor:discreteconst}
	Under the assumptions of Theorem \ref{convdiscrete}, if $X$ is spherically symmetric then for $\alpha>0$ we have that 
	\[ 		\lim_{n \to \infty}   \PP \Big(\widetilde{T}_{n, {N_u, N_c, c_0}}  > q_{1-\alpha}(\mathbb G) \Big)  =\alpha,   
	\]
	where as above $q_{1-\alpha}(\mathbb G)$ is the $1-\alpha$-Quantile of $\Vert  \mathbb G \Vert_{\mathcal{F}}$. 

	In contrast, if $X$ is not spherically symmetric then 
	\[ 		\lim_{n \to \infty}   \PP \Big(\widetilde{T}_{n, {N_u, N_c, c_0}}  > q_{1-\alpha}(\mathbb G) \Big) =  1.   
	\]
\end{cor}

The bootstrap version of the discretized test statistic is given by 
\begin{eqnarray*}
	\widetilde{T}^*_{n, N_u, N_c, c_0}  =  \Vert  \mathbb G^*_n \Vert_{\mathcal{F}_{N_u, N_c, c_0}},
\end{eqnarray*}
where $\mathbb G^*_n$ is defined in \eqref{Gnstar} and $\mathcal{F}_{N_u, N_c, c_0}$  in \eqref{DN}.
%, it follows from Theorem \ref{CondEmpProc} and a conditioning argument that also 
%
%\begin{eqnarray}\label{LimitTilde1}
%	\widetilde{T}^*_{n, N_u, N_c, c_0}    \Rightarrow  \widetilde{T}_{N_u, N_c, c_0}   \stackrel{d}{=}  \Vert \mathbb{G} \Vert_{\mathcal{D}_{N_u, N_c, c_0}}
%\end{eqnarray}
%given almost all sequences $X_1, X_2, \ldots$. 
%
The results in Corollary \ref{cor:discreteconst} extend to $\widetilde{T}^*_{n, N_u, N_c, c_0}$. We refrain from providing the formal details, which would require an extension of Theorem \ref{CondEmpProc} to the situation of changing functions classes as in Theorem \ref{convdiscrete}.

Algorithm \ref{algo1} describes how to simulate from the distribution of $\widetilde{T}^*_{n, N_u, N_c, c_0}$, conditionally on $X_1, \ldots, X_n$, to obtain a critical value. 

%
%
% These new data are then used to compute new realizations of $\widetilde{T}_{n, N_u, N_c, c_0}$ and also the empirical quantile $q^B_{\alpha, n, N_u, N_c, c_0}$ based on $B$ bootstrap samples, for some large integer $B > 0$.\\
%
%\par \noindent The bootstrap procedure for producing $\widetilde{T}^{1*}_{n, N_u, N_c, c_0}, \ldots, \widetilde{T}^{B*}_{n, N_u, N_c, c_0}$ and computing  $q^B_{\alpha, n, N_u, N_c, c_0}$  
%
%\bigskip

\begin{algorithm}[h!]%\label{alg:firstalgo}
	
	\KwData{$X_1, \ldots, X_n$, $N_u$, $N_c$, $c_0$ and $B$}
	
	\medskip
	\KwResult{$\widetilde{T}^{1*}_{n, N_u, N_c, c_0}, \ldots, \widetilde{T}^{B*}_{n, N_u, N_c, c_0}$}
	
	\medskip
	
	$b =0$ \;
	\While{$b  < B$ }{
		
		\medskip
		$b \gets   b +1$ \;
		
		\medskip
		draw $W^b_1, \ldots, W^b_n$  i.i.d. $\sim \mathcal{U}(S_{d-1})$; 
		
		\medskip
		sample with replacement $R^{b*}_1, \ldots, R^{b*}_n$ from $\{ \Vert X_1 \Vert, \ldots, \Vert X_n \Vert \}$
		independently of $(W^b_1, \ldots, W^b_n)$;
		
		\medskip
		define $X^{b*}_1:=R^{b*}_1  \  W_1, \ldots,   X^{b*}_n:=R^{b*}_n  \  W_n$;

		\medskip
		draw $U^b_1, \ldots, U^b_{N_u}, V^b_1, \ldots, V^b_{N_u} $ i.i.d. $\sim \mathcal{U}(\mathcal{S}_{d-1})$;
		
		\medskip
		compute $\widetilde{T}^{b*}_{n, N_u, N_c, c_0}$ in the same way as $\widetilde{T}_{n, N_u, N_c, c_0}$ replacing
		$(X_1, \ldots, X_n)$ with $(X^{b*}_1, \ldots, X^{b*}_n)$ and $(U_1, \ldots, U_{N_u}, V_1, \ldots, U_{N_u})$ with $(U^b_1, \ldots, U^b_{N_u},  V^b_1, \ldots, V^b_{N_u} )$;

		\medskip
		store    $\widetilde{T}^{b*}_{n, N_u, N_c, c_0}$;
	}
	compute the empirical $(1-\alpha)$-quantile $q^B_{1-\alpha, n, N_u, N_c, c_0}$ based on the sample $(\widetilde{T}^{1*}_{n, N_u, N_c, c_0}, \ldots, \widetilde{T}^{B*}_{n, N_u, N_c, c_0})$.
	
	\medskip
	\caption{Bootstrap procedure for producing $\widetilde{T}^{b*}_{n, N_u, N_c, c_0}, b=1, \ldots, B$ and the empirical quantile $q^B_{1-\alpha, n, N_u, N_c, c_0}$. }
	\label{algo1}
	
\end{algorithm}

\section{Extensions to testing elliptical symmetry}\label{sec:extension}

Elliptical symmetry is tightly connected to spherical symmetry. If the random vector $X \in \RR^d$ has a non-singular covariance matrix $\Sigma_0$ and mean $\mu_0$, then it is elliptically symmetric if and only if the representation
\begin{equation}\label{eq:reprellip}
	X  = \mu_0 + \Sigma_0^{1/2} Y
\end{equation}
holds true for a spherically symmetric $Y$. In the following, we consider testing the hypothesis 
\begin{eqnarray*}
	H_0': X \ \textrm{is elliptically symmetric} \  \quad  \textrm{versus} \ \quad H_1': X \ \textrm{is not.}
\end{eqnarray*}
We shall restrict ourselves to investigating the asymptotic distribution of the proposed test statistics under the null hypothesis, and hence will always assume that $X$ is elliptically symmetrically distributed, and $Y$ will correspond to the representation in \eqref{eq:reprellip}.  
We let $X_1, \ldots, X_n$ be i.i.d.~with the distribution of $X$ and denote the associated empirical probability measure by $\PP_n$. We shall write $P$ for the distribution of $X$. If $(f_a)$ is a family of $P$-integrable functions, and $\hat a_n$ depends on $X_1, \ldots, X_n$, we write 
\[ P f_{\hat a_n} : = E[f_{\hat a_n}(X)]:= \int_{\RR^d}\, f_{\hat a_n}(x)\, \dd P(x),\]
a random variable depending on $X_1, \ldots, X_n$. As above, $ \mathbb{G}_n  f = \sqrt n (\PP_n - P) f$ denotes the empirical process. Finally, $\PP$ and $\Ex$ are probability and expected value operator on the space where $X, X_1, X_2, \ldots $ are defined.

\subsection{Convergence of the empirical process}\label{sec:convprop}

Consider the empirical mean and covariance matrix
% If one knew , then testing elliptical symmetry is clearly equivalent to testing spherical symmetry. In more realistic frameworks, these parameters are unknown and hence have to be estimated. The most natural estimators to consider are of course 
%
\begin{eqnarray*}
	\bar{X}_n = n^{-1}  \sum_{i=1}^n  X_i \  \  \ \textrm{and} \ \ \ \widehat{\Sigma}_{n}  = n^{-1}  \sum_{i=1}^n (X_i - \bar{X}_n)\,  (X_i - \bar{X}_n)^\top.
\end{eqnarray*}
In the following we shall assume that $X$ is absolutely continuously distributed, in which case $\widehat{\Sigma}_{n}$ is non-singular for $n \geq d+1$ with probability $1$, see \cite{gupta1971nonsingularity}.
% In any case, by consistency, $\widehat{\Sigma}_{n}$ will become non-singular as $n \to \infty$ with probability $1$. 
%
Therefore, roughly speaking we can reduce tests for elliptical symmetry to testing for spherical symmetry of the standardized random variables
\begin{eqnarray}\label{Trans}
\widehat{X}_i =  \widehat{\Sigma}^{-1/2}_n  (X_i -  \bar{X}_n), \qquad i  =1, \ldots, n. 
\end{eqnarray}
We start with the following theorem, which is based on results from \cite{vdvwell2007} for dealing with empirical processes involving estimated functions. 
We let $\eta = (A, \mu) \in \mathcal{E}$, where $A$ ranges through positive-definite $d \times d$-matrices, $\mu \in \RR^d$, and as before we let 
\[ \theta = (u,v,c)\in \Theta= \mathcal{S}_{d-1} \times \mathcal{S}_{d-1} \times \RR.\]
Consider functions on $\RR^d$ defined by
\begin{eqnarray}\label{eq:fthetaeta}
f_{\theta,\eta}(x) = (v - (v^\top u)\, u)^\top A \, (x - \mu)\,  \mathds{1}_{u^\top  A (x - \mu)  \ge c},
\end{eqnarray}
where $\theta = (u,v,c) \in \Theta$ and $\eta = (A, \mu)\in \mathcal{E}$. Recall the empirical process $\mathbb{G}_n$ from \eqref{eq:charac}.
\begin{theorem}\label{Th: empestim}
	If $X$ is absolutely continuously distributed with $\Ex[\|X\|^2]< \infty$, setting 
	\[ 		\eta_0 =  \big(\Sigma_0^{-1/2}, \mu_0\big) \quad  \text{ and } \quad \widehat{\eta}_n =  (\widehat{\Sigma}_n^{-1/2}, \bar{X}_n)\]
	we have as $n \to \infty$ that
	\begin{eqnarray*}
		\sup_{\theta \in \Theta} \Big \vert \mathbb{G}_n  (f_{\theta,\widehat{\eta}_n } -  f_{\theta,\eta_0})  \big \vert \to  0 \qquad \text{	in probability.}
	\end{eqnarray*}
\end{theorem}
If $X$ is elliptically symmetric, then $P f_{\theta, \eta_0}=0$ for all $\theta \in \Theta$, and hence an asymptotic test for elliptical symmetry can again be based on a Kolmogorov-Smirnov type statistic
\[ T_{n,e} = \sup_{\theta \in \Theta} \, \big| \sqrt n \, \mathbb P_n f_{\theta, \widehat \eta_n}\big|. \]
%
%To describe the weak convergence of the empirical process under elliptical symmetry, we define the class of functions 
%\begin{eqnarray*}
%	\mathcal{F}'  \equiv \left  \{f =  f_{\theta, \eta}=  f_{u, v, c, \eta}:  (u, v, c)  \in \Theta, \eta \in \mathcal{B}(\eta_0,1/2) \right\}
%\end{eqnarray*}
%with , and with $\eta_0 = (\Sigma^{-1/2}_0, \mu_0)$
%\begin{eqnarray*}
%	\mathcal{B}(\eta_0, r )  = \left\{  (A, \mu) \in H_0  : ||| A -  \Sigma^{-1/2}_0 ||| +  | \mu- \mu_0 | < r \right\}
%\end{eqnarray*}
%for some $r > 0$, and where we recall that $H_0  = \{\eta = (A, \mu) \in \mathbb{R}^{d \times d} \times \RR: A  \  \ \textrm{is positive definite}  \}$.  
%
However, estimation of the parameters in $\eta_0$ induces an additional drift term in the limit process, as described in the following result. 
\begin{cor}\label{Weakconvresc}
	Suppose that  $X$ is elliptically symmetric such that $\Ex[\Vert X \Vert^4 ] < \infty$ and that $X$ admits a continuous Lebesgue density. Furthermore, assume that the density $g$ of any coordinate of $Y  = \Sigma^{-1/2}_0 (X- \mu_0)$ satisfies $\sup_{v \in \RR} \vert v \vert  g(v)  < \infty$. Then,  as $n \to \infty$
	\begin{eqnarray*}
		\{ \sqrt n \mathbb P_n f_{\theta, \widehat \eta_n}\mid \theta \in \Theta\}   \Rightarrow \ \{\mathbb{G} f_{\theta, \eta_0}  +   \mathbb{L}(\theta)\mid \theta \in \Theta\},
	\end{eqnarray*}
	in the space $\ell^\infty(\Theta)$, a Gaussian process in which $\mathbb G$ is defined in Theorem \ref{Weakconv}, and 
	\begin{eqnarray}\label{Drift}
	\mathbb{L}(\theta) 	& = &  - (w^\top \Sigma_0^{-1/2}\, \mathbb{D}) \, \PP(V > c) + \, w^\top \big(\mathbb{S} \Sigma^{1/2}_0  + \Sigma^{1/2}_0\mathbb{S} \big) u \, \Ex[V \mathds{1}_{V > c}], \nonumber \\
	&& 
	\end{eqnarray}
	where $w= v  - (v^\top u)\, u$, $V$ is distributed as any coordinate of $Y$ in \eqref{eq:reprellip}, and $ \mathbb D$ and $\mathbb{S} $ are as in Lemma \ref{ConvSigma}. 
\end{cor}

%{\color{red} The joint distribution of $ \mathbb{S}$ and $\mathbb Z$ and $\mathbb G$ is required in this result.}

\subsection{Modifying the spherically symmetric bootstrap}\label{sec:bootelli}

Corollary \ref{Weakconvresc} shows that we cannot simply apply the spherically symmetric bootstrap from Section \ref{sec:spherboot} to the standardized variables $\widehat{X}_i$. Rather, as in \cite{koltchi2000} the bootstrapped sample needs to be standardized again before computing the Kolmogorov-Smirnov statistic.
More precisely, sample with replacement from the empirical distribution of the norms of $\widehat{R}_i = \Vert \widehat{X}_i \Vert, i =1, \ldots, n $ to obtain $\widehat{R}^*_1, \ldots, \widehat{R}^*_n$. Then, draw a random sample  $(W_1, \ldots, W_n)$  from the unit sphere $\mathcal{S}_{d-1}$, independently of $(\widehat{R}^*_1, \ldots, \widehat{R}^*_n)$ and set
\begin{eqnarray*}
	\widetilde{X}^*_i  =  \widehat{R}^*_i \,  W_i \quad \text{and} \quad \widehat{X}^*_i =  \widehat{\Sigma}^{*\, -1/2}_n \, (\widetilde{X}^*_i -  \bar{X}^*_n),
\end{eqnarray*}
where $\bar{X}^*_n$ and $\widehat{\Sigma}^*_n$ are empirical mean and covariance matrix of $\widetilde{X}^*_1, \ldots, \widetilde{X}^*_n$. 
The following theorem implies consistency of this modified procedure we have just described.
\medskip
\begin{theorem}\label{TheoModifBoot}
	%Let $\theta = (u, v, c)  \in  \Theta = \mathcal{S}_{d-1} \times  \mathcal{S}_{d-1} \times \mathbb R$. Also, let us denote by $f^0_\theta(x)  =  (v - (v^\top u) u)^\top x \mathds{1}_{u^\top x \ge c}$, 
	Let $\widehat{\mathbb P}^*_n$ be the empirical process based on $\widehat{X}^*_1, \ldots, \widehat{X}^*_n$. Then,  under the same conditions of Corollary \ref{Weakconvresc}, we have conditionally of $X_1, X_2, \ldots$ that
	\begin{eqnarray*}
		\{\sqrt n \widehat{\mathbb P}^*_n f_\theta \mid \theta \in \Theta\} \Rightarrow \{\mathbb G f_{\theta, \eta_0}  + \mathbb L(\theta)\mid \theta \in \Theta\}
	\end{eqnarray*}
	in the space $\ell^\infty(\Theta)$. Here, $\mathbb G$ and the drift $\mathbb L$ are  the same Gaussian process and drift as in Corollary \ref{Weakconvresc},  $f_\theta = f_{u,v,c}$ is defined in \eqref{classF} and $f_{\theta,  \eta_0}$ in \eqref{eq:fthetaeta}. 
\end{theorem}
We present the main steps of the proof in the Appendix.
Together with a discretization as in Section \ref{sec:discretealg} this results in  Algorithm \ref{algo2}.  

\begin{algorithm}[h!]
	\KwData{$X_1, \ldots, X_n$, $N_u$, $N_c$, $c_0$ and $B$}
	
	\medskip
	\KwResult{$\widetilde{T}^{1*}_{n, N_u, N_c, c_0}, \ldots, \widetilde{T}^{B*}_{n, N_u, N_c, c_0}$}
	
	\medskip
	
	compute the empirical mean $\bar{X}_n$ and empirical covariance matrix $\widehat \Sigma_n$;
	
	\medskip
	
	compute the standardized variables  $\widehat{X}_1, \ldots, \widehat{X}_n$: $\widehat X_i =  \widehat \Sigma^{-1/2}_n  (X_i - \bar{X}_i)$;
	
	$b =0$ \;
	\While{$b  < B$ }{
		
		\medskip
		$b \gets   b +1$ \;
		
		\medskip
		draw $W^b_1, \ldots, W^b_n$  i.i.d. $\sim \mathcal{U}(S_{d-1})$; 
		
		\medskip
		sample with replacement $\widehat{R}^{b*}_1, \ldots, \widehat{R}^{b*}_n$ from $\{ \Vert \widehat{X}_1 \Vert, \ldots, \Vert \widehat{X}_n \Vert \}$
		independently of $(W^b_1, \ldots, W^b_n)$;
		
		\medskip
		define $\widetilde{X}^{b*}_1:=\widetilde{R}^{b*}_1  \  W_1, \ldots,  \widetilde{X}^{b*}_n:=\widetilde{R}^{b*}_n  \  W_n$;

		\medskip
		
		standardize $\widetilde{X}^{b*}_1, \ldots, \widetilde{X}^{b*}_n$ resulting in $\widehat{X}^{b*}_1, \ldots, \widehat{X}^{b*}_n$;
		
		\medskip
		draw $U^b_1, \ldots, U^b_{N_u}, V^b_1, \ldots, V^b_{N_u} $ i.i.d. $\sim \mathcal{U}(\mathcal{S}_{d-1})$;
		
		\medskip
		compute $\widetilde{T}^{b*}_{n, N_u, N_c, c_0}$ in the same way as described in Algorithm \ref{algo1};

		\medskip
		store    $\widetilde{T}^{b*}_{n, N_u, N_c, c_0}$;
	}
	compute the empirical $(1-\alpha)$-quantile $q^B_{\alpha, n, N_u, N_c, c_0}$ based on the sample $(\widetilde{T}^{1*}_{n, N_u, N_c, c_0}, \ldots, \widetilde{T}^{B*}_{n, N_u, N_c, c_0})$.
	
	\medskip
	\caption{Bootstrap procedure for producing $\widetilde{T}^{b*}_{n, N_u, N_c, c_0}, b=1, \ldots, B$ and the empirical quantile $q^B_{\alpha, n, N_u, N_c, c_0}$ for the extended test to elliptical symmetry. }
	\label{algo2}
\end{algorithm}

\section{Simulations}\label{sec: simss}

In this section we  present the results of an extensive simulation study. We focus on testing for spherical symmetry, and in Section \ref{sec:typeI} investigate the type I error for the bootstrap test proposed in Section \ref{sec:discretealg}, while in Section \ref{sec:typeII} we investigate its power under various alternatives, and compare it with the procedures of \cite{Baringhaus91} and \cite{Liang2008}. 
In Section \ref{sec:compareHenze}, we consider the alternatives presented in \cite{Henze2014} for a power comparison of their procedure with our test. 
Finally, in Section \ref{sec:ellipsymm} we briefly consider the finite sample performance of our test for elliptical symmetry, and compare it with six methods which were investigated in \cite{sakhanenko2008testing}. 
%
%simulation results which we obtained in different settings.  Our first goal is to check that  under spherical symmetry the test has a type I error that stays asymptotically close to the targeted level $\alpha$. Then,   we compute the power for some selected distributions which are not spherically symmetric. 

All simulations were performed for dimensions $d \in \{3, 6, 10\}$ and with the choice of $N_u=1000$, $N_c=500$ and $c_0=10$. For estimating the unknown critical point of our statistic, $B=100$ bootstrap replications were used. 
For the nominal level we always chose $\alpha = 0.05$.

\subsection{Type I Error}\label{sec:typeI}
Table \ref{tab:level} displays the empirical levels of the test for spherical symmetry in Section \ref{sec:discretealg} obtained with five spherically symmetric distributions  and with $1000$ replications. These null distributions are as follows:
\begin{itemize}
	\item\lq\lq  G\rq\rq: Multivariate Gaussian with mean 0 and correlation matrix equal to the identity $I_d$,
	\item \lq\lq Cauchy \rq\rq: Cauchy distribution with location parameter 0 and scale parameter 1,
	\item \lq\lq  MVt, df \rq\rq: Multivariate $ t $-distribution with $df$ degrees of freedom,
	\item \lq\lq  Kotz rq\rq: Kotz type distribution with $N=2$, $r=1$ and $s=0.5$,
	\item \lq\lq  PVII \rq\rq: Pearson type VII distribution with $N=10$, $m=2$.
\end{itemize}
For the definitions of the multivariate $t$-distribution, the Kotz and Pearson type VII distributions we refer to \cite{Fang1990}.  The results show that for all distributions and dimensions we are near the theoretical level of 0.05, except for the Cauchy distribution.  The latter does not have a finite expectation, and hence our method is not applicable in this case.  Note also that neither the dimension nor the sample size seem to strongly influence the type I error.\\

\begin{table}[!h]
	\caption{\normalsize{Finite-sample level of the proposed bootstrap-test for spherical symmetry for nominal level $\alpha = 0.05$. 
			%Parameters  of the test are $c_0=10, N_c = 500$ and $N_u=1000$. 
			See the text for description of the parameters of the test, and for details on the distributions under which data are generated. 			
	}}
	\label{tab:level}
	\centering
	\begin{tabular}[t]{cc|ccc|}
		\toprule
		\multicolumn{1}{c}{ Distr.} & \multicolumn{1}{c}{$ n$} & \multicolumn{1}{c}{Dim.~$ 3$} & \multicolumn{1}{c}{Dim.~$ 6$} & \multicolumn{1}{c}{Dim.~$ 10$} \\
		\cmidrule(l{2pt}r{2pt}){1-2} \cmidrule(l{2pt}r{2pt}){3-3} \cmidrule(l{2pt}r{2pt}){4-4} \cmidrule(l{2pt}r{2pt}){5-5}
		& 100    & 0.053 & 0.057 & 0.05\\
		\multirow{-2}{*}{\raggedright\arraybackslash {G}} & 200  & 0.063 & 0.056 & 0.067\\
		\cmidrule{1-5}
		& 100    & 0.003 & 0.008 & 0.009\\
		
		\multirow{-2}{*}{\raggedright\arraybackslash Cauchy} & 200    & 0.015 & 0.008 & 0.005\\
		\cmidrule{1-5}
		& 100    & 0.048 & 0.047 & 0.052\\
		
		\multirow{-2}{*}{\raggedright\arraybackslash {MVt, $df=5$}} & 200    & 0.055 & 0.045 & 0.047\\
		\cmidrule{1-5}
		& 100    & 0.065 & 0.062 & 0.051\\
		
		\multirow{-2}{*}{\raggedright\arraybackslash Kotz} & 200    & 0.063 & 0.064 & 0.063\\
		\cmidrule{1-5}
		& 100    & 0.061 & 0.054 & 0.059\\
		
		\multirow{-2}{*}{\raggedright\arraybackslash PVII} & 200    & 0.058 & 0.062 & 0.062\\
		\bottomrule
	\end{tabular}
\end{table}
\subsection{Type II Error}\label{sec:typeII}
Next, to study the power properties of our test for spherical symmetry in Section \ref{sec:discretealg}, we consider the following distributions:
\begin{itemize}
	\item \lq\lq G,  $\rho$\rq\rq: Multivariate Gaussian with mean $\mu = 0$ and correlation matrix $\Sigma$, where $\Sigma_{ij}=\rho$ if $i\ne j$ and $\Sigma_{ii}=1$, with $\rho\in\{0.4,0.6\}$,
	\item \lq\lq MG,  $\mu$\rq\rq: Mixture of Gaussian distributions with mean $\mu$ and $-\mu$ and correlation matrix $I_d$. An observation from this distribution has probability 0.5 to be sampled from  $\mathcal{N}(\mu,I_d)$, otherwise, it is sampled from  $\mathcal{N}(-\mu,I_d)$. The $\mu$ chosen are $(\mu_1,0,\ldots,0)$, with $\mu_1\in\{1,1.5,2\}$,
	\item \lq\lq NCG,  $\mu$\rq\rq: Multivariate Gaussian $\mathcal{N}(\mu,I_d)$ with mean $\mu =(\mu_1,0,\ldots,0)$, $\mu_1 \in \{ 1,2 \}$. \textit{NC} in the abbreviation stands for \lq\lq not centered \rq\rq,
	\item \lq\lq MTt\rq\rq: Meta-Type normal distribution obtained from a multivariate t-distribution with 5 degrees of freedom. Definitions and theory on Meta-Type distributions can be found in \cite{Fang2002} or \cite{Liang2008},
	\item \lq\lq Cube \rq\rq: Uniform distribution on the Hypercube $[-1,1]^d$.
\end{itemize}

{\small
	\begin{table}[h]
		\caption{ \normalsize{Finite-sample power of the proposed bootstrap-test $T_n$ for spherical symmetry, together with results for two competing procedures $A_n$ and $B_n$, see text for details. }}
		\label{tab:power}
		\centering
		\begin{tabular}[t]{cc|ccc|ccc|ccc|}
			\toprule
			\multicolumn{1}{c}{ } & \multicolumn{1}{c}{ } & \multicolumn{3}{c}{Dim.~$ 3$} & \multicolumn{3}{c}{Dim.~$ 6$} & \multicolumn{3}{c}{Dim.~$ 10$} \\
			\cmidrule(l{2pt}r{2pt}){3-5} \cmidrule(l{2pt}r{2pt}){6-8} \cmidrule(l{2pt}r{2pt}){9-11}
			Distr. & $n$ & $T_n$ & $A_n$ & $B_n$ & $T_n$ & $A_n$ & $B_n$ & $T_n$ & $A_n$ & $B_n$\\
			\midrule
			& 100 & 0.772 & 0.142 & 0.504 & 0.986 & 0.328 & 0.941 & 0.998 & 0.484 & 0.998\\
			
			\multirow{-2}{*}{\raggedright\arraybackslash {G,\ }{$ \rho=0.4$}} & 200 & 0.991 & 0.235 & 0.928 & 1 & 0.541 & 1 & 1 & 0.732 & 1\\
			\cmidrule{1-11}
			& 100 & 1 & 0.615 & 0.969 & 1 & 0.928 & 1 & 1 & 0.995 & 1\\
			
			\multirow{-2}{*}{\raggedright\arraybackslash {G,\ }{$\rho=0.6$}} & 200 & 1 & 0.854 & 1 & 1 & 0.996 & 1 & 1 & 1 & 1\\
			\cmidrule{1-11}
			& 100 & 0.34 & 0.063 & 0.264 & 0.191 & 0.063 & 0.188 & 0.096 & 0.063 & 0.142\\
			
			\multirow{-2}{*}{\raggedright\arraybackslash {MG, \ }{$\mu_1=1$}} & 200 & 0.761 & 0.071 & 0.596 & 0.481 & 0.06 & 0.414 & 0.191 & 0.058 & 0.234\\
			\cmidrule{1-11}
			& 100 & 0.976 & 0.315 & 0.95 & 0.883 & 0.137 & 0.847 & 0.574 & 0.102 & 0.644\\
			
			\multirow{-2}{*}{\raggedright\arraybackslash {MG, \ }{$\mu_1=1.5$}} & 200 & 1 & 0.528 & 1 & 1 & 0.165 & 1 & 0.975 & 0.124 & 0.981\\
			\cmidrule{1-11}
			& 100 & 1 & 0.067 & 1 & 1 & 0.056 & 1 & 1 & 0.063 & 1\\
			
			\multirow{-2}{*}{\raggedright\arraybackslash {NCG, \ }{$\mu_1=1$}} & 200 & 1 & 0.06 & 1 & 1 & 0.053 & 1 & 1 & 0.047 & 1\\
			\cmidrule{1-11}
			& 100 & 1 & 0.913 & 1 & 1 & 0.566 & 1 & 1 & 0.369 & 1\\
			
			\multirow{-2}{*}{\raggedright\arraybackslash {NCG, \ }{$\mu_1=2$}} & 200 & 1 & 0.998 & 1 & 1 & 0.833 & 1 & 1 & 0.594 & 1\\
			\cmidrule{1-11}
			& 100 & 0.053 & 0.21 & 0.056 & 0.07 & 0.413 & 0.071 & 0.051 & 0.651 & 0.05\\
			
			\multirow{-2}{*}{\raggedright\arraybackslash {MTt \ }{$df=5$}} & 200 & 0.068 & 0.349 & 0.062 & 0.042 & 0.727 & 0.062 & 0.053 & 0.901 & 0.062\\
			\cmidrule{1-11}
			& 100 & 0.086 & 0.818 & 0.07 & 0.067 & 0.998 & 0.062 & 0.057 & 1 & 0.07\\
			
			\multirow{-2}{*}{\raggedright\arraybackslash Cube} & 200 & 0.102 & 0.985 & 0.074 & 0.063 & 1 & 0.066 & 0.05 & 1 & 0.072\\
			\bottomrule
		\end{tabular}
	\end{table}
}

The results can be found in Table \ref{tab:power} under the columns \lq\lq $T_n$\rq\rq. The outcomes from two other spherical symmetry tests, labeled with $A_n$ and $B_n$, are also integrated into the table. The first one, $A_n$, denotes the test introduced in \cite{Liang2008}, %that employs the statistic $A_n$ with symmetric discrepancy
while the second test, $B_n$, represents the test proposed by \cite{Baringhaus91}.  Our test has the best performance in dimension $d=3$ for the first seven distributions.   Under the mixture of Gaussian distribution, MG, the test looses power as the dimension increases. This is mainly caused by the decreasing importance of the deviation.   For the Meta-type normal and cubic distributions neither our test nor the test by \cite{Baringhaus91} exhibit substantial power above the level. However, additional simulations (not displayed) show that the power tends to $1$ with increasing sample size.%, see Table \ref{tab:nlarge}.

\subsection{Some comparison based on \cite{Henze2014}}\label{sec:compareHenze}
Next, we investigate power properties of the tests for spherical symmetry under the following alternatives from the simulation study in \cite{Henze2014},
\begin{itemize} 
	\item \lq\lq$H_1^{(2)}$\rq\rq: the distribution of the random vector $X=(X_1,X_2,X_3)$ such that $X_1$, $X_2$ and $X_3$ are independent,  $X_1$ and $X_2$ are standard Gaussian, while $X_3$ has an Exponential distribution with rate 1.
	
	\item \lq\lq$H_1^{(4)}$\rq\rq: the distribution of the random vector $X=(X_1,X_2,X_3)$ such that $X_1$ and $X_2$ are generated from a uniform distribution on an equilateral triangle centered at the origin, where the length of each side is set to $\sqrt{12}$, while  $X_3$ is independent of $X_1$ and $X_2$ and has a uniform distribution on the interval $(-\sqrt{12},\sqrt{12})$.
\end{itemize}
Table \ref{tab:henze} gives the power properties of the tests for these alternatives for sample sizes $n \in \{100, 200 \}$. In the column \lq\lq Henze et al \rq\rq, we display the largest power from the family of tests (various tuning parameters, Kolmogorov-Smirnov and Cramer-von-Mises statistic) as obtained by \cite{Henze2014} in their simulation study. 
%
%
%for  different alternatives distributions in dimension $d=3$ and . The results are shown for their Kolmogorov-Smirnov and Cramer-von-Mises tests and different choices of the parameters $\mathcal{R}$, $w$, $a$ and $\rho$. For a detailed description of these parameters on which depend the definition of the test statistics, we refer to Section 1 and 2 in \cite{Henze2014}. In the column  \ of Table \ref{tab:henze} we give the largest power values obtained by the authors for $n=100$ and $n=200$  among all the considered tests. 

For $n=100$, the necessary test of \cite{Liang2008} based on the statistic $A_n$ has the lowest power. For $n=200$, the four tests have comparable performances with slight advantage for the test based on $T_n$ in the case of  $H^{(4)}_1$.

\begin{table}[!h]
	\begin{center}
		\caption{\normalsize Power of the tests based on $T_n$, $A_n$, $ B_n$ and the (best) test of \cite{Henze2014} for additional alternative hypotheses. % \textcolor{red}{Caption above and usual Table format (Isaia)}
		}
		\label{tab:henze}
		\begin{tabular}[t]{cc|c c c c|}
			\hline
			& & & & &  \\[-2ex]
			Distr. & n & $T_n$ &  $A_n$ & $B_n$    &  Henze et al. 
			\\ 
			\hline
			& $100$ & 1  & 0.177  &  1  &   0.73 \\ [-0ex]
			\multirow{-2}{*}{\raggedright\arraybackslash {$H_1^{(2)}$ }}  & $200$ &  1 & 0.304 & 1   & 1   \\
			\hline
			& $100$ & 1 & 0.993  & 1 & 0.72  \\ [-0ex]
			\multirow{-2}{*}{\raggedright\arraybackslash {$H_1^{(4)}$ }}  & $200$  & 1 & 1 & 1  & 0.99   \\
			\hline
		\end{tabular}
	\end{center}
\end{table}

\subsection{Simulations for testing elliptical symmetry}\label{sec:ellipsymm}

In this section we conduct a small simulation study for the bootstrap test for elliptical symmetry as suggested in Algorithm \ref{algo2} in Section \ref{sec:bootelli}. 
To investigate the level, we simulated from the distributions \lq\lq G, $\rho$\rq\rq \ with $\rho \in \{0, 0.4, 0.6\}$ and \lq\lq Kotz\rq\rq \   under the null hypothesis.
%,  and   \lq\lq MG, $\mu_1 $\rq\rq \ with $\mu_1 \in \{1, 1.5\} $, \lq\lq Exps\rq\rq, \lq\lq Prisma\rq\rq and \lq\lq Cube\rq\rq \ under the alternative hypothesis, where  
%\begin{itemize}
%	\item \lq\lq Exps\rq\rq: the first two components  of $X \sim  \text{Exps}$  are generated from the exponential distribution with rate $0.5$ and $2$. All the other $d-2$ components are sampled from an exponential distribution with fixed rate $d$.
%	%
%	\item \lq\lq Prisma\rq\rq:  the first two components of $X \sim  \text{Prisma}$  are generated from a Uniform distribution on an equilateral triangle centered at the origin, where the length of each side is set to $\sqrt{12}$, all the other components are independent of $X_1$,$X_2$ and uniform distributed on the interval $(-\sqrt{6},\sqrt{6})$
%\end{itemize}
%
The results for sample sizes $n \in \{100, 200\}$ under the null hypothesis are shown in Table \ref{tab:ellsym}.
%, and under the alternatives in Table \ref{tab:ellsymalt}.  
%
%
\begin{table}[!h]
	\centering
	\caption{\normalsize Rejection rates for testing elliptical symmetry under the null hypothesis for nominal level of $\alpha = 0.05$.}
	\label{tab:ellsym}
	
	\begin{tabular}[t]{cc|ccc|}
		\toprule
		Distr. & $n$ & Dim.~$ 3$ & Dim.~$ 6$ & Dim.$ 10$\\
		\midrule
		
		& 100 & 0.061 & 0.058 & 0.068\\
		
		\multirow{-2}{*}{\raggedright\arraybackslash G, $\rho=0$} & 200 & 0.066 & 0.069 & 0.063\\
		\cmidrule{1-5}
		& 100 & 0.065 & 0.057 & 0.069\\
		
		\multirow{-2}{*}{\raggedright\arraybackslash G, $\rho=0.4$} & 200 & 0.051 & 0.054 & 0.064\\
		\cmidrule{1-5}
		& 100 & 0.065 & 0.052 & 0.071\\
		
		\multirow{-2}{*}{\raggedright\arraybackslash G, $\rho=0.6$} & 200 & 0.067 & 0.053 & 0.082\\
		\cmidrule{1-5}
		& 100 & 0.068 & 0.069 & 0.056\\
		
		\multirow{-2}{*}{\raggedright\arraybackslash Kotz} & 200 & 0.068 & 0.07 & 0.063\\

		\bottomrule
	\end{tabular}
\end{table}

%
%\begin{table}[!h]
%	\centering
%	\caption{\normalsize {\color{red} Leave this table out, since we have two other tables?} Rejection rates for testing elliptical symmetry under the selected alternatives for nominal level of $\alpha = 0.05$.
%	{\color{red} Corrected Mistakes (Isaia)}
%}
%	\label{tab:ellsymalt}
%	
%	\begin{tabular}[t]{cc|ccc|}
%		\toprule
%		Distr. & $n$ & Dim.~$ 3$ & Dim.~$ 6$ & Dim.$ 10$\\
%		\midrule
%		& 100 & {\color{red}0.066} & {\color{red}0.057} & {\color{red}0.067}\\
%		
%		%\multirow{-2}{*}{\raggedright\arraybackslash MVt} & 200 & 0.108 & 0.171 & 0.207\\
%		%\cmidrule{1-5}
%		%& 100 & 0.066 & 0.057 & 0.067\\
%		
%		\multirow{-2}{*}{\raggedright\arraybackslash {MG, \ }{$\mu_1=1$}} & 200 & 0.099 & 0.061 & 0.071\\
%		\cmidrule{1-5}
%		& 100 & 0.197 & 0.074 & 0.073\\
%		
%		\multirow{-2}{*}{\raggedright\arraybackslash {MG, \ }{$\mu_1=1.5$}} & 200 & 0.479 & 0.104 & 0.062\\
%		\cmidrule{1-5}
%		& 100 & 1 & 0.993 & 0.756\\
%		
%		\multirow{-2}{*}{\raggedright\arraybackslash Exps} & 200 & 1 & 1 & 0.999\\
%		\cmidrule{1-5}
%		& 100 & 0.946 & 0.285 & 0.092\\
%		
%		\multirow{-2}{*}{\raggedright\arraybackslash Prisma} & 200 & 1 & 0.795 & 0.152\\
%		\cmidrule{1-5}
%		& 100 & 0.398 & 0.136 & 0.082\\
%		
%		\multirow{-2}{*}{\raggedright\arraybackslash Cube} & 200 & 0.778 & 0.321 & 0.098\\
%		\bottomrule
%	\end{tabular}
%\end{table}
%
%write some conclusions here

Given the moderate sample sizes, the actual levels are reasonably close to the nominal one of $\alpha = 0.05$. 
%The power of the test is not very high for the mixture of Gaussians, which can be explained by the fact that the components are not very well separated. For $d=3$ and $6$, and $n =200$, we obtain a very good power for \lq\lq Exps\rq\rq \ and \lq\lq Prisma \rq\rq. The \lq\lq Cube\rq\rq \ distribution remains challenging as it was in the case of  testing for spherical symmetry. 
%
%
To investigate the power, we shall compare our test to those studied in \cite{sakhanenko2008testing}. There, three tests from the class of tests of \cite{koltchi2000}, which are denoted by $S_n$, $C_n$ and $H_n$, are compared with the tests proposed by \cite{beran1979testing, huffer2007test} and \cite{manzotti2002statistic}, denoted $B_n^0$, $Q_n$ and $HP_n$. 
As in \cite{sakhanenko2008testing} we consider the following alternative distributions in $2$ and $3$ dimensions: 
\begin{itemize}
	\item $M_2$: a mixture of two bivariate Gaussian distributions: the first one a standard normal and the second with mean $\mu_2=(1,2)$ and covariance matrix $\Sigma_2=(5,-4;-4,5)$, with mixing probabilities equal to $1/2$,
	\item $M_3$: a mixture of two three-dimensional Gaussian distributions: the first again a standard normal, the other with mean $\mu=(1,2,3)$  and covariance $\Sigma=(5,-4,1;-4,6,-4;1,-4,5)$, with mixing probabilities equal to $1/2$,
	\item $\Gamma + N_{d-1}$: a $d$-dimensional random vector, with the first component a $\Gamma(2,3)$ distribution and the other $d-1$ coordinates independent of the first one and distributed according to a $(d-1)$ - dimensional standard normal distribution,
	\item $B(\beta)$: a multivariate Burr distribution, which has the distribution of $(1+X/Y)^{-\beta}$, where $Y$ is a univariate gamma distribution with shape parameter 1 and scale parameter $\beta$, that we'll set equal to 0.5 in both dimension 2 and 3, while the coordinates of $X$ are all mutually independent and identically gamma-distributed with shape parameter $2$ and scaling parameter $3$,
	\item $U(C)$: the uniform distribution on the unit cube in dimensions 2 and 3  (this corresponds to \lq\lq Cube\rq\rq \ used in Section 4.2),
	\item $U(A)$: the uniform distribution on the set $A=[0,1) \times [0,\pi/2)\cup[0,1) \times [\pi,3\pi/2)\cup[1,\sqrt{2})) \times [\pi/2,\pi)\cup[1,\sqrt{2} )\times [3\pi/2,2\pi)$ in dimension 2. 
\end{itemize}
We run our test with the parameters $N_u = 1000$, $N_c = 500$ and $c_0 = 10$ as above, but the number of replications is set to be $5000$ to match the setting in \cite{sakhanenko2008testing}.

The results obtained for our test $T_n$, together with those from the paper \cite{sakhanenko2008testing} for the tests mentioned above, are given in Table \ref{tab:power1} for dimension $2$ and in Table \ref{tab:power2} for dimension $3$. The test $T_n$ compares very favorably to the competing procedures in terms of power.  

{\small
	\begin{table}[h]
		\caption{ \normalsize{Rejection rates for testing elliptical symmetry under the selected two-dimensional alternatives for nominal level of $\alpha = 0.05$. The table contains the results from \cite[Table 2]{sakhanenko2008testing}.}}
		\label{tab:power1}
		\centering
		\begin{tabular}[t]{cc|ccccccc|}
			\toprule
			\multicolumn{1}{c}{ } & \multicolumn{1}{c}{ } & \multicolumn{7}{c}{Dim.~$ 2$} \\
			\cmidrule(l{2pt}r{2pt}){3-9}
			Distr. & $n$ & $T_n$ & $S_n$ & $C_n$ & $H_n$ & $B^0_n$ & $Q_n$ & $HP_n$ \\
			\midrule
			& 50 & 0.575 & 0.333 & 0.241 & 0.258 & 0.060 & 0.295 & 0.103 \\
			
			& 100 & 0.922 & 0.728 & 0.490 & 0.510 & 0.101 & 0.609 & 0.261 \\
			
			\multirow{-3}{*}{\raggedright\arraybackslash {$M_2$}} & 200 & 0.999 & 0.984 & 0.880 & 0.869 & 0.478 & 0.936 & 0.623 \\
			\cmidrule{1-9}
			& 50 & 0.316 & 0.095 & 0.447 & 0.494 & 0.061 & 0.077 & 0.111\\
			
			& 100 & 0.591 & 0.172 & 0.809 & 0.860 & 0.115 & 0.117 & 0.285 \\
			
			\multirow{-3}{*}{\raggedright\arraybackslash {$\Gamma + N_1$}} & 200 & 0.928 & 0.375 & 0.991 & 0.996 & 0.350 & 0.203 & 0.692 \\
			\cmidrule{1-9}
			& 50 & 0.204 & 0.091 & 0.058 & 0.068 & 0.051 & 0.116 & 0.029 \\
			
			& 100 & 0.410 & 0.159 & 0.058 & 0.086 & 0.058 & 0.221 & 0.039 \\
			
			\multirow{-3}{*}{\raggedright\arraybackslash {U(C)}}& 200 & 0.756 & 0.343 & 0.059 & 0.166 & 0.098 & 0.454 & 0.042 \\
			\cmidrule{1-9}
			& 50 & 0.521 & 0.487 & 0.250 & 0.338 & 0.081 & 0.419 & 0.152 \\
			
			& 100 & 0.910 & 0.839 & 0.508 & 0.699 & 0.358 & 0.746 & 0.363 \\
			
			\multirow{-3}{*}{\raggedright\arraybackslash {B(0.5) }} & 200 & 1 & 0.991 & 0.906 & 0.979 & 0.930 & 0.975 & 0.737 \\
			\cmidrule{1-9}
			& 50 & 0.977 & 0.199 & 0.222 & 0.194 & 0.423 & 0.074 & 0.536\\
			
			& 100 & 1 & 0.547 & 0.399 & 0.413 & 0.950 & 0.069 & 0.824 \\
			
			\multirow{-3}{*}{\raggedright\arraybackslash {U(A)}} & 200 & 1 & 0.961 & 0.726 & 0.843 & 1 & 0.078 & 0.998 \\
			\bottomrule
		\end{tabular}
	\end{table}

	{\small
		\begin{table}[h]
			\caption{ \normalsize{Rejection rates for testing elliptical symmetry under the selected three-dimensional alternatives for nominal level of $\alpha = 0.05$. The table contains the results from \cite[Table 3]{sakhanenko2008testing}.}}
			\label{tab:power2}
			\centering
			\begin{tabular}[t]{cc|cccccc|}
				\toprule
				\multicolumn{1}{c}{ } & \multicolumn{1}{c}{ } & \multicolumn{6}{c}{Dim.~$ 3$} \\
				\cmidrule(l{2pt}r{2pt}){3-8}
				Distr. & $n$ & $T_n$ & $S_n$ & $B^1_n$ & $B^2_n$ & $Q_n$ & $HP_n$ \\
				\midrule
				
				& 100 & 1 & 0.997 & 0.058 & 0.076 & 0.954 & 0.995  \\
				
				\multirow{-2}{*}{\raggedright\arraybackslash {$M_3$}} & 200 & 1 & 1 & 0.052 & 0.079 & 1 & 1  \\
				\cmidrule{1-8}
				
				& 100 & 0.551 & 0.198 & 0.056 & 0.064 & 0.135 & 0.188  \\
				
				\multirow{-2}{*}{\raggedright\arraybackslash {$\Gamma + N_1$}} & 200 & 0.932 & 0.414 & 0.054 & 0.078 & 0.260 & 0.444  \\
				\cmidrule{1-8}
				
				& 100 & 0.399 & 0.317 & 0.052 & 0.059 & 0.179 & 0.033  \\
				
				\multirow{-2}{*}{\raggedright\arraybackslash {U(C)}}& 200 & 0.797 & 0.711 & 0.058 & 0.068 & 0.375 & 0.036 \\
				\cmidrule{1-8}
				
				& 100 & 0.938  & 0.972 & 0.096 & 0.145 & 0.855 & 0.686  \\
				
				\multirow{-2}{*}{\raggedright\arraybackslash {B(0.5) }} & 200 & 1 & 1 & 0.116 & 0.228 & 0.994 & 0.969  \\
				
				\bottomrule
			\end{tabular}
		\end{table}

		\newpage

%\section{Some conclusions}
%In this paper, we construct a new Kolmogrov-Smirnov-type test for spherical symmetry based on the characterization given by \cite{eaton86} in his Theorem 1.  A bootstrap procedure is used to approximate the critical region for  any given  level $\alpha \in (0,1)$, and is shown to yield asymptotically the true critical region. We also show that the resulting test is consistent for any fixed alternative. As similar bootstrap procedure was already considered in the work of \cite{DiksTong99} and also in \cite{koltchi2000} for testing spherical symmetry. Our proof,  which is based on empirical processes arguments gives an alternative way of showing  bootstrap consistency in this nonparametric problem, under the only assumption that the underlying distribution has a finite second moment. In extending these ideas to testing elliptical symmetry  by standardizing the sample we found that the same bootstrap procedure cannot be used as is based on these pseudo-observations. Indeed, our Theorem \ref{Th: empestim}  and its subsequent Corollary \ref{Weakconvresc} show that such a procedure gives rise to a limit which is different from the one obtained by considering the zero-centered empirical process based on the original data. The difference between the two limits is due to the presence of a drift, the main responsible for the bootstrap consistency in this problem. We show numerically that this inconsistency can be fixed by mimicking the nature of the pseudo-observations and thus by standardizing the bootstrapped vectors. 

\section*{Acknowledgments.}  \ The second author would like to thank Jon A. Wellner for pointing out very relevant references on testing spherical symmetry and for providing us with the manuscript by \cite{vdvwell2007}.   
All authors thank very much two anonymous referees for their very constructive comments, which led to a much improved manuscript.
\section{Appendix: Proofs } \label{empprocess}

\subsection{Proofs for Section \ref{sec: KStests}}

\begin{proof}[Proof of Lemma \ref{Theo.generate}]
	First we show that a random vector  $X$ in  $\RR^d$ with zero expectation is spherically symmetric if and only if we have that
	\begin{equation}\label{eq:eaton1}
	\Ex[(v  - (v^\top u) u)^\top X |  u^\top X ]  = 0 \qquad \forall \ u,v \in \mathcal{S}_{d-1}.
	\end{equation}
	Indeed, \eqref{eq:eaton1} immediately implies \eqref{eq:eatonchara}, and the converse is equally clear since $v  - (v^\top u) u  $ and $u$ are perpendicular. 
	Now Lemma \ref{Theo.generate} follows from the characterization \eqref{eq:eaton1} and the following lemma. 
\end{proof}
\begin{lemma}\label{LemInterm}
	Let $Y,Z$ be random variables with $E|Y|< \infty$. Then $\Ex[Y \mathds{1}_{Z \geq c}] = 0$ for all $c \in \RR$ if and only if $\Ex[Y|Z]=0$.  
\end{lemma}
\begin{proof}[Proof of Lemma \ref{LemInterm}]
	If $\Ex[Y|Z]=0$ then $\Ex[Y \mathds{1}_{Z \geq c}] = \Ex[\mathds{1}_{Z \geq c}\Ex[Y | Z]] = 0$ for every $c$. Conversely, $\Ex[Y|Z]=0$ if $\Ex[Y\, \mathds{1}_{Z \in B}]=0$ for all Borel-subsets $B$ of $\RR$. But here it suffices to consider sets $B$ in an intersection-stable generator (which contains $\Omega$ resp.~$\RR$).
	Now, the collection of sets $\{ [c,\infty), \ c \in \RR \}$ is an intersection-stable generator of the Borel-$\sigma$-field of  $\RR$, and hence the events $\{ \{ Z \in [c,\infty)\}, \ c \in \RR \} = \{ \{Z \geq c\}, \ c \in \RR\}$ are an intersection-stable generator of the $\sigma$-field generated by $Z$.
	By letting $c \to - \infty$ und using dominated convergence, from the assumption we obtain that $\Ex[Y] = \Ex[Y\, \mathds{1}_\Omega]=0$, so we can add the full space $\Omega$ to this system of sets.  
\end{proof}

\bigskip

\begin{proof}[Proof of Theorem \ref{Weakconv}] 
	We shall show that $\mathcal{F}$  in (\ref{classF}) is a VC-class of functions with an integrable envelope, see \cite[Theorem 2.5.2 and Section 2.6]{vdvwell}.
	%It is enough to show that the class is $P$-Donsker.   
	We have that %$\mathbb{I}_{\Vert x \Vert \le M } \mathbb{I}_{v^\top x_M \ge \alpha}  = \mathbb{I}_{\Vert x \Vert \le M } \mathbb{I}_{v^\top x \ge \alpha}$, it follows that 
	\begin{align}\label{ClassG}
	\mathcal{F} \subset \mathcal{G}& : =  \Big \{g(x) =w^\top x \  \mathds{1}_{u^\top x \ge c} \mid  \notag \\
	& \hspace{2cm}   \ (u, w) \in \RR^d \times \RR^d,   c \in \RR    \Big \},
	%& = & \mathcal{G}_1  \mathcal{G}_2 \mathcal{G}_3 
	\end{align}
	%with
	%\begin{eqnarray*}
	%\mathcal{G}_1: = \Big \{ g_1: g_1(x) = u^\top x, \  x \in \RR^d, u \in \mathcal{S}_{d-1}   \Big \},
	%\end{eqnarray*}
	%\begin{eqnarray*}
	%\mathcal{G}_2: = \Big \{ g_2: g_2(x) =  \mathbb{I}_{\Vert x \Vert \le M}, \   x \in \RR^d, M \in (0, \infty)   \Big \},
	%\end{eqnarray*}
	%and 
	%\begin{eqnarray*}
	%\mathcal{G}_3: = \Big \{ g_3: g_3(x) =  \mathbb{I}_{v^\top x \ge c}, \   x \in \RR^d,  v \in \mathcal{S}_{d-1},  c \in \RR   \Big \}.
	%\end{eqnarray*}
	and it suffices to show the assertions for $\mathcal{G}$. 
	First, the function $F(x)  = 2 \Vert x \Vert $ is an envelope for  $\mathcal{G}$ which is square-integrable by assumption. 
	We show now that $\mathcal{G}$ is a VC-class; i.e., we need to show that the collection of all subgraphs 
	\begin{equation*}
	\mathcal{C} = \Big\{  \big \{(x, t)  \in \RR^{d+1} \mid g(x)  > t  \big \} \mid g \in \mathcal{G} \Big\}
	\end{equation*}
	forms a VC-class  in $\mathbb{R}^{d+1}$; see  Section 2.6 in \cite{vdvwell}.  We have that 
	\begin{eqnarray*}
		\mathcal{C}&=  &   \Big \{ \{ (x, t) \in \mathbb{R}^{d+1} \mid   w^\top x  > t, u^\top x \ge c\} \mid  c \in \RR, \ u, w \in \RR^d  \Big \}  \\
		& & \ \  \cup  \  \Big \{\{ (x, t) \in \mathbb{R}^{d+1} \mid   t  < 0, u^\top x  <  c\} \mid  c \in \RR,\ u \in \RR^d \Big \} \\
		& = &  \mathcal{C}_1  \cup \mathcal{C}_2.
	\end{eqnarray*}
	From \cite[Lemma 2.6.27]{vdvwell} it suffices to show that $\mathcal{C}_i$ are VC-classes of sets, which follows from \cite[Section II.4, Lemma 18]{pollard2012convergence}.
\end{proof}

\begin{proof}[Proof of Proposition \ref{cons}]
	From Theorem \ref{Weakconv} it follows that  as $n \to \infty$, 
	\begin{eqnarray*}
		\Vert \mathbb{G}_n \Vert_{\mathcal{F}} =\Vert \sqrt n (\PP_n - P) f\Vert_{\mathcal{F}}  \to_d \Vert \mathbb G \Vert_\mathcal{F}.
	\end{eqnarray*}
	% Now, for $f = f_{u, v,c}  \in \mathcal{F}$ we have 
	%\begin{eqnarray*}
	%	\mathbb Pf_{u, v, c}  &=   & \Ex[(v- v^\top u u)^\top X \mathds{1}_{u^\top X \ge c} ]
	%\end{eqnarray*}
	%%Note that 
	%%\begin{eqnarray*}
	%%\Big   \vert v^\top X_M \mathbb{I}_{u^\top X \ge \epsilon}  -  v^\top X \mathds{1}_{u^\top X \ge \epsilon} \Big \vert \le \Vert X \Vert \mathbb{I}_{\Vert X \Vert > M }  \to_{a.s.}  0, \ \  \  \textrm{as $M \to \infty$}.
	%%\end{eqnarray*}
	%%Combining the preceding convergence with $ \vert v^\top X_M \mathbb{I}_{u^\top X \ge \epsilon}  \vert \le \Vert X \Vert$ which admits a finite expectation we conclude that  $E [v^\top X_M \mathbb{I}_{u^\top X \ge \epsilon} ] \to E [v^\top X \mathbb{I}_{u^\top X \ge \epsilon} ]$ as $M \to \infty$, uniformly in $u, v$ and $\epsilon$. This in turn implies that we can find $M_0$, $u_0, v_0$ and $\epsilon_0$ such that
	%%\begin{eqnarray*}
	%%\Big \vert \Ex[v^T_0 X_{M_0} \mathbb{I}_{u^T_0 X \ge \epsilon_0} ] \Big \vert \ge  \delta_0/2
	%%\end{eqnarray*}
	%%where $\delta_0$ is taken from (\ref{D}).
	%so that  
	%\begin{eqnarray*}
	%	\sqrt n \Vert \mathbb P f \Vert_{\mathcal{F}} \ge  \sqrt{n} \delta_0.
	%\end{eqnarray*}
	On the other hand, we have that
	\begin{eqnarray*}
		T_n = \sqrt n \Vert \PP_n f \Vert_{\mathcal{F}}  = \sqrt n \Vert( \PP_n  -  P)f  +  P  f \Vert_{\mathcal{F}}  \ge  \sqrt n \Vert  P f \Vert_{\mathcal{F}} - \sqrt n \Vert (\PP_n  - P)f  \Vert_{\mathcal{F}}.
	\end{eqnarray*}
	Thus,
	\begin{eqnarray*}
		\PP(T_n > q_{1-\alpha}(\mathbb G))  & \ge    &   \PP\Big(\sqrt n \Vert  P  f \Vert_{\mathcal{F}}  -  \sqrt n \Vert (\PP_n  - P)f  \Vert_{\mathcal{F}}   > q_{1-\alpha}(\mathbb G)  \Big)  \\
		& \ge & \PP\Big( \sqrt n \Vert (\PP_n  - P)f  \Vert_{\mathcal{F}}   < -q_{1-\alpha}(\mathbb G)  + \sqrt n\,  \Delta_0(X) \Big)  \to 1
	\end{eqnarray*}
	as $n \to \infty$, where $\Delta_0(X)>0 $ is as in \eqref{D}. 
\end{proof}
\begin{proof}[Proof of Theorem \ref{CondEmpProc}]
	We shall use \cite[Theorem 2.11.1]{vdvwell}. To this end, we shall check that the assumptions of this theorem hold conditionally on the sample $X_1, X_2, \ldots $ almost surely, and that the limit process always is $\mathbb G$. 
	%
	% and for that we will first check that its conditions are all satisfied  (note that in the current setting we do not have measurability issues
	We write $L_0$ for the distribution of $R$, $L^*_n$ for  the empirical distribution $ n^{-1}\sum_{j=1}^n \delta_{R_i} $  and $\mathcal{U} $ for the uniform distribution $ \mathcal{U}(\mathcal{S}_{d-1})$. When taking the expected value conditionally on $X_1, X_2, \ldots$, we write $\Ex_{L^*_n \otimes \mathcal{U}}$, and $\Ex_{\mathcal{U}}$ if the expression only involves the $W_j$ and not the $R^*_j$. 
	We write 
	\[ Z_{ni}  = n^{-1/2}  \delta_{W_i R^*_i}, \qquad i =1, \ldots, n.\]
	Since $W_i R^*_i$ is spherically symmetric under $L^*_n \otimes \mathcal{U}$, we have that \linebreak $\Ex_{L^*_n \otimes \mathcal{U}}[f(W_i R^*_i)]=0$ for $f \in \cF_s$ by assumption, so that 
	\begin{eqnarray*}
		\mathbb{G}^*_n(f)   =   \sum_{i=1}^n \left(Z_{ni}(f)  - \Ex_{L^*_n \otimes \mathcal{U}} Z_{ni}(f)\right)  =  \sum_{i=1}^n Z_{ni}(f), \qquad f \in \cF_s,
	\end{eqnarray*}
	Consider the variance semimetric
	\begin{eqnarray*}
		\rho^2(f, g)   &  =    &    \Ex_{L_0 \otimes \mathcal{U}} \big[\left(f(R  W)  -  g(R  W)  \right)^2\big]  =    \Ex_{P} \big[\left(f(X)  -  g(X  )\right)^2\big]
		%&  = &   \frac{1}{n}  \sum_{i=1}^n E_{\mathcal{U}}  \left(  f(r_i W)  -  g(r_i W)  \right)^2
	\end{eqnarray*}
	for $f,g \in \cF_s$. 
	By the assumptions on the class of functions $\cF_s$, it is $P$-Donsker (for any $P$ for which the majorant is square-integrable), and hence by  \cite[Corollary 2.3.12 ]{vdvwell} it follows that  the space  $(\cF_s, \rho)$ is totally bounded. %. In fact, if we defined $Q  = L^*_n \otimes \mathcal{U}$, it follows from Theorem 2.6.7 of \cite{vdvwell} that  the covering number
	%\begin{eqnarray*} 
	%N(\epsilon \Vert F \Vert_{Q, 2}, \mathcal{D}_1(\underline{u}_N, \underline{v}_N),  L_2(Q))  =  N(\epsilon \Vert F \Vert_{Q, 2}, \mathcal{D}_1(\underline{u}_N, \underline{v}_N), \rho)  < \infty
	%\end{eqnarray*}
	%for any $\epsilon \in (0,1)$, where $F(x)  =  2  \Vert x \Vert \le 2$ and $\Vert F\Vert^2_{Q, 2}  \le 4 n^{-1}  \sum_{i=1}^n r^2_i  < \infty$ almost surely  as it is assumed that $ \int \Vert x \Vert^2  dP(x) < \infty$.  
	Given $\eta > 0$ we have that
	\begin{equation*}
	\sum_{i=1}^n \Ex_{L^*_n \otimes \mathcal{U}} \left[  \Vert Z_{ni} \Vert^2_{\mathcal{F}}  \mathds{1}_{\{  \Vert Z_{ni} \Vert_{\mathcal{F} }  > \eta \}} \right] 
	\leq  \Ex_{L^*_n \otimes \mathcal{U} } \left[ F(R^* W)^2  \mathds{1}_{\{ \vert F (R^* W) \vert   > n^{1/2} \eta\}} \right],
	\end{equation*}
	where $F$ is the envelope of $\cF_s$ and $(R^*,W)$ is distributed as $L^*_n \otimes \mathcal{U}$. Expanding the expected value w.r.t.~$L^*_n$ yields
	\begin{equation*}
	\Ex_{L^*_n \otimes \mathcal{U} } \left[ F^2(R^* W)  \mathds{1}_{\{ \vert F (R^* W) \vert   > n^{1/2} \eta\}} \right] = \frac{1}{n}\, \sum_{j=1}^n \Ex_{ \mathcal{U} } \left[ F^2(R_j W)  \mathds{1}_{\{ \vert F (R_j W) \vert   > n^{1/2} \eta\}} \right].
	\end{equation*}
	Since $\Ex [ F^2(R\, W)]< \infty$ by assumption, from the dominated convergence theorem is follows that 
	\[ \Ex\big[ \Ex_{\mathcal{U}}[ F^2(R W) \mathds{1}_{\{ \vert F (R W) \vert   > n^{1/2} \eta\}}]\big]  \to 0, \quad n \to \infty,\]
	and the strong law implies that as $n \to \infty$,
	\begin{align*}
	& \sum_{i=1}^n \Ex_{L^*_n \otimes \mathcal{U}} \left[  \Vert Z_{ni} \Vert^2_{\mathcal{F}}  \mathds{1}_{\{  \Vert Z_{ni} \Vert_{\mathcal{F} }  > \eta} \right] \\
	\leq & \ \frac{1}{n}\, \sum_{j=1}^n \Ex_{ \mathcal{U} } \left[ F(R_j W)^2  \mathds{1}_{\{ \vert F (R_j W) \vert   > n^{1/2} \eta\}} \right] \to 0.
	\end{align*}
	This shows the first assumption in \cite[Theorem 2.11.1]{vdvwell}. 
	
	As for the second, we need to show that for any sequence $\delta_n >0$ with $\delta_n \to 0$ we have that 
	\begin{align}
	& \sup_{ \rho(f, g)  < \delta_n }   \sum_{i=1}^n \Ex_{L^*_n \otimes \mathcal{U}}  \left[ \left( Z_{ni} (f)  - Z_{ni}(g)\right)^2\right]  \nonumber \\
	& =   \sup_{ \rho(f, g)  < \delta_n }  \frac{1}{n} \sum_{i=1}^n  \Ex_{\mathcal{U}} \left[ \left( f(R_i  W)  -g (R_i  W)\right)^2 \right]    \to  0 \label{eq:cond2vdV} 
	%& = &  \sup_{ \rho(f, g)  < \delta_n }  \rho^2(f, g)  < \delta_n^2  \to 0
	\end{align}
	for almost all $X_1, X_2, \ldots $, where $f,g \in \cF_s$ in the supremum. To this end, we show below that the class of functions defined as 
	\begin{equation}\label{eq:GlivCant}
	\mathcal{C}:=  \Big \{r \mapsto \Ex_{\mathcal{U}} [(f- g)^2(W r) ] \mid  r \in (0, \infty), \   \textrm{and} \  f, g \in \mathcal{F}_s \Big \}
	\end{equation}
	is Glivenko-Cantelli for the law $L_0$ of $R$. Then we have in particular that
	\begin{align}\label{eq:GlivCant}
	\sup_{ \rho(f, g)  < \delta_n}  \Big \vert \sum_{i=1}^n  \Ex_{\mathcal{U}} \left[ \left( f(R_i  W)  -g (R_i  W)\right)^2 \right]  -  E_{L_0 \otimes \mathcal{U}} \left[  (f-g)^2 (WR)\right]    \Big \vert \to 0
	\end{align}
	almost surely as $n \to \infty$. Since   
	\[ E_{L_0 \otimes \mathcal{U}} \left[  (f-g)^2 (WR)\right]   =  \rho^2(f, g)\]
	we evidently have 
	%
	%   and $ E_{L^*_n  \otimes \mathcal{U} }  \left[  (f-g)^2 (WR)\right]   =    \frac{1}{n} \sum_{i=1}^n  E_{\mathcal{U}} \left[ f(r_i  W)  -g (r_i  W)\right]^2  $  it follows that
	%
	\begin{eqnarray*}
		%	&& \sup_{ \rho(f, g)  < \delta_n}  \frac{1}{n} \sum_{i=1}^n  E_{\mathcal{U}} \left[ f(r_i  W)  -g (r_i  W)\right]^2   \\
		&& \sup_{ \rho(f, g)  < \delta_n}  \Big \vert  E_{L_0 \otimes \mathcal{U}} \left[  (f-g)^2 (WR)\right]  \to 0,
	\end{eqnarray*}
	which together with \eqref{eq:GlivCant} implies \eqref{eq:cond2vdV}. 
	To show that 
	\eqref{eq:GlivCant} is a Glivenko-Cantelli class, we shall apply \cite[Theorem 2.4.3]{vdvwell}. We first note that $\tilde{F}(r) =  4 \, \Ex_{\mathcal{U}} [F^2(W r) ]  $ is an $L_0$-integrable envelope for $\mathcal{C}$. 
	To bound the covering numbers, for functions $f,f_i,g,g_j$, smaller in absolute value than $F$, for given $M>0$ we estimate
	\begin{align*}
	& \Ex_{L_n^* \otimes \mathcal{U}}\Big[\Big \vert  (f - g)^2(WR^*)  \mathds{1}_{\tilde{F}(R^*) \le M } -  (f_i - g_j)^2(WR^*)   \mathds{1}_{\tilde{F}(R^*) \le M }\Big \vert \Big]  \\
	%	& =  \Ex_{Q \otimes \mathcal{U}}\Big[\Big \vert  (f - g)^2(WR)  -  (f_i - g_j)^2(WR) \Big \vert   \mathds{1}_{F(R) \le \sqrt{M} } \Big]  \\
	%&& = E_{Q \otimes \mathcal{U}}\Big[\Big \vert  (f - g)^2(WR)  \mathds{1}_{F(R) \le \sqrt{M} } -  (f_i - g_j)^2(WR)  \mathds{1}_{F(R) \le \sqrt{M} } \Big \vert \Big]\\
	=&   \Ex_{L_n^* \otimes \mathcal{U}}\Big[\Big \vert \left(f(WR^*) -  f_i(WR^*)  -  g(WR^*)  + g_j(WR^*) \right)  \\ 
	& \  \hspace{2cm}    (f(WR^*) + f_i(WR^*) -  g(WR^*)  - g_j(WR^*))\, \Big \vert \mathds{1}_{F(WR^*) \le \sqrt{M}/2 } \Big]  \\
	%	& =  \Ex_{Q \otimes \mathcal{U}}\Big[\Big \vert \left(f(WR) -  f_i(WR)  +  g(WR)  - g_j(WR) \right)  \\ 
	%	& \  \hspace{2cm}    (f(WR) + f_i(WR)  +  g(WR)  + g_j(WR)\Big \vert \mathds{1}_{F(R) \le \sqrt{M} } \Big]  \\
	\le & 2\, \sqrt{M}   \Big( \Ex_{L_n^* \otimes \mathcal{U}}\Big[\Big \vert f(WR^*) -  f_i(WR^*)  \Big \vert  \Big ] +    \Ex_{Q \otimes \mathcal{U}} \Big[\Big \vert    g(WR^*)  - g_j(WR^*)  \Big \vert  \Big ]\Big).
	\end{align*}
	This implies for the covering numbers of the class $\mathcal{C}_M  = \big\{ h \mathds{1}_{\{ \tilde{F} \le M \}} \mid  h \in \mathcal{C}  \big \}$ that for $\epsilon>0$, 
	\[ N\big(4 \sqrt{M} \epsilon, \mathcal{C}_M, L_2(L_n^*)  \big) \leq N^2\big(\epsilon, \mathcal F_s, L_2(L_n^* \otimes U)  \big).\]
	Since by assumption, $\mathcal F_s$ is a VC-class of functions, we obtain the estimate
	\[ N\big(\epsilon, \mathcal{C}_M, L_2(L_n^*)  \big) \leq K\, \Big(\frac{4 \sqrt{M} \| \tilde F\|_{L_n^*,1}}{\epsilon} \Big)^V \] 
	for positive constants $K,V>0$. By integrability of $\tilde F$ under $L_0$ and the law of large numbers, this implies that $\log N\big(\epsilon, \mathcal{C}_M, L_2(L_n^*)  \big) = o_{L_0}(n)$, as required in \cite[Theorem 2.4.3]{vdvwell}.

	To check the third condition  of \cite[Theorem 2.11.1]{vdvwell}, since $\cF_s$ is a VC-class, there are constants $K',V'>0$ for which  
	\begin{eqnarray*}
		N\left(\epsilon   ,  \mathcal{F}, L_2(Q_n)  \right)  \le K'  \left( \frac{\Vert F\Vert_{Q_n,2}}{\epsilon}  \right)^{V'},
	\end{eqnarray*}
	where $Q_n$ is the empirical distribution of $(W_1, R^*_1), \ldots, (W_n , R^*_n)$. The distance in $L_2(Q_n)$ is 
	\[ d^2_n (f, g)  = \sum_{i=1}^n (Z_{ni}(f) - Z_{ni}(g))^2,\]
	as required in \cite[p.~206]{vdvwell}. Hence we estimate the entropy integral as
	%
	%for some constants $K > 0$ and $\gamma > 0$.  Here, $\Vert F\Vert^2_{Q_n,2}   = 4  n^{-1} \sum_{i=1}^n r^2_i  < \infty$ almost surely. Now, since 
	%
	% as defined , it follows that there exists some finite constant $K' > 0$ such that 
	\begin{align*}
	& \int_0^{\delta_n}  \sqrt{\log N\left(\epsilon,  \mathcal{F}, d_n  \right) }  \dd\epsilon\\
	\le  & \int_0^{\delta_n}  \sqrt{ \log(K')  + V'\, \log(\Vert F\Vert_{Q_n,2}) + V'\, \log(1/\epsilon)  }  d\epsilon  \\
	\le &  \delta_n\, \sqrt{ \log(K')  + V'\, \log(\Vert F\Vert_{Q_n,2})}\, + V'\, \int_0^{\delta_n}   \sqrt{\log(1/\epsilon)}  \dd\epsilon   \\
	= & \delta_n\, \sqrt{ \log(K')  + V'\, \log(\Vert F\Vert_{Q_n,2})}\,  +  \sqrt{2 \gamma}   \int_{1/\delta_n}^\infty  \frac{\sqrt{\log t}}{t^2} \dd t \\
	\to  &\  0, \  \qquad  \textrm{as $\delta_n \to 0$}.
	\end{align*}
	Finally, for the limiting covariance we have that
	\begin{eqnarray*}
		\cov_{L^*_n \otimes \mathcal{U}} \left(\sum_{i=1}^n Z_{ni}(f),  \sum_{i=1}^n Z_{ni}(g)\right ) & =  &  n \cov_{L^*_n \otimes \mathcal{U}}(Z_{n1}(f), Z_{n1}(g)) \\
		& = &    \Ex_{L^*_n \otimes \mathcal{U}}  \big[f(W R^*) g(W R^*)\big]  \\
		&= &  \frac{1}{n}  \sum_{i=1}^n E_{\mathcal{U}}  \left[f(W R_i) g(W R_i)  \right]    \\
		&  \to &  \Ex_{L_ 0\otimes \mathcal{U}}  \left[ f(WR)  g(WR) \right] 
	\end{eqnarray*}
	almost surely by the strong law of large numbers. The latter equals \linebreak  $\Ex\left[ f(X)  g(X) \right]  = \cov(f(X), g(X))$, the covariance of the limiting process  $\mathbb{G}$. We conclude that $\mathbb{G}^*_n$ also converges weakly to this process, given almost all sequences $X_1, X_2, \ldots $.  
\end{proof}

\begin{proof}[Proof of Theorem \ref{convdiscrete}]
	We shall apply the changing classes central limit theorem,  Theorem 19.28 in \citet{vdv98}, conditionally on almost all sequences  $ U_1 , U_{2}, \ldots$ and $V_1, V_{2}, \ldots$. As the limiting process will be the same almost surely, the convergence then is also unconditional. 
	
	In the proof  of Theorem \ref{Weakconv} we showed that $ \mathcal{F}_{\Theta}$ is a VC-class of functions, and hence the same is true for each 	
	$ \mathcal{F}_{\Theta}^{(n)} = \{ f_{n(\theta)} \mid \theta \in \Theta\}$ with the same bound on the entropy numbers. Therefore, the condition on the entropy integral in \citet[Theorem 19.28]{vdv98} is satisfied. 
	Further, $\|X\|_2$ is a majorant for $ \mathcal{F}_{\Theta}$ and hence also for each function class $ \mathcal{F}_{\Theta}^{(n)}$, which satisfies the Lindeberg-condition by the assumed integrability of $\|X\|_2^2$. Finally, for each $\theta \in \Theta$, by continuity of the distribution of $X$ we have that $f_{n(\theta)}(x) \to f_\theta(x)$ for almost all $x \in \RR^d$, conditionally on almost all sequences $ U_1 , U_{2}, \ldots$ and $V_1, V_{2}, \ldots$. By dominated convergence we obtain convergence of the covariances $P(f_{n(\theta)}\, f_{n(\tilde \theta)}) - P(f_{n(\theta)})\, P(f_{n(\tilde \theta)}) \to P(f_{\theta}\, f_{\tilde \theta}) - P(f_{\theta})\, P(f_{\tilde \theta})$, conditionally on almost all sequences $ U_1 , U_{2}, \ldots$ and $V_1, V_{2}, \ldots$. 
\end{proof}
\begin{proof}[Proof of Corollary \ref{cor:discreteconst}]
	Again we set $ \mathcal{F}_{\Theta}^{(n)} = \{ f_{n(\theta)} \mid \theta \in \Theta\}$. 
	First suppose that $X$ is spherically symmetric. Then by Theorem \ref{convdiscrete}, since $P f = 0$, $f \in  \mathcal{F}_{\Theta}$,
	\[ \widetilde{T}_{n, {N_u, N_c, c_0}} = \| \mathbb G_n\|_{\mathcal{F}_{\Theta}^{(n)}} \Rightarrow \Vert  \mathbb G \Vert_{\mathcal{F}_\Theta},\]
	which implies the first statement by continuity of the distribution of $\Vert  \mathbb G \Vert_{\mathcal{F}_\Theta}$. 
	
	Now suppose that $X$ is not spherically symmetric, so that in \eqref{D}, $\Delta_0(X)>0$. By continuity of the distribution of $X$ we have that 
	\[ \Vert  P  f \Vert_{\mathcal{F}_{\Theta}^{(n)}}  \uparrow \Delta_0(X)\] 
	along almost all sequences $ U_1 , U_{2}, \ldots$ and $V_1, V_{2}, \ldots$. Let $\mathcal{A}_{U,V}$ denote the $\sigma$-algebra generated by $ U_1 , U_{2}, \ldots$ and $V_1, V_{2}, \ldots$. 
	Then arguing as in the proof of Proposition \ref{cons}, 
	\begin{align*}
	& \, \PP\big(\widetilde{T}_{n, {N_u, N_c, c_0}} > q_{1-\alpha}(\mathbb G) \mid \mathcal{A}_{U,V} \big)\\
	\ge    &  \,  \PP\Big(\sqrt n \Vert  P  f \Vert_{\mathcal{F}_{\Theta}^{(n)}}  -  \sqrt n \Vert (\PP_n  - P)f  \Vert_{\mathcal{F}_{\Theta}^{(n)}}   > q_{1-\alpha}(\mathbb G) \mid \mathcal{A}_{U,V} \Big)  \to 1 
	\end{align*}
	almost surely, since from the proof of Theorem \ref{convdiscrete} we also have convergence of $\sqrt n \Vert (\PP_n  - P)f  \Vert_{\mathcal{F}_{\Theta}^{(n)}}$ to $\| \mathbb G\|_{\mathcal{F}_{\Theta}}$ conditionally on $\mathcal{A}_{U,V}$. Taking the expected value of the conditional probability yields the statement of the corollary. 
\end{proof}

\subsection{Proofs for Section \ref{sec:extension}}

\begin{proof}[Proof of Theorem \ref{Th: empestim}]
	As discussed in the beginning of Section \ref{sec:convprop}, we have that 	$\PP(\widehat \eta_n \in \mathcal{E})  = 1$, $n \geq d+1$. 
	%Put $(u, v, c) = \theta  \in \Theta= \mathcal{S}^2_{d-1} \times \RR$. Also, let
	%\begin{eqnarray*}
	%H  = \Big \{\eta =  (A, \mu):   A \in \RR^{d \times d}, \mu  \in \RR \Big \}
	%\end{eqnarray*}
	%and 
	%\begin{eqnarray*}
	%H_0   = \Big \{\eta =  (A, \mu) \in H:   A  \ \textrm{is positive definite }  \Big \}.
	%\end{eqnarray*}
	We shall use Theorem 2.1 in \cite{vdvwell2007}, and for that we need to show  that 
	\begin{eqnarray}\label{Cond1}
	\sup_{\theta \in \Theta}  P(f_{\theta, \widehat{\eta}_n}  -  f_{\theta,\eta_0} )^2  \to  0
	\end{eqnarray}
	in probability as $n \to \infty$ and that the class 
	\begin{eqnarray}\label{Donskf}
	\Big \{ f_{\theta,\eta} \mid  \  \   (\theta, \eta) \in \Theta  \times \mathcal{E}  \Big \}  \  \  \textrm{is \ $P$-Donsker}.
	\end{eqnarray}
	To show (\ref{Donskf}), we write
	\begin{eqnarray*}
		f_{\theta,\eta}(x)  & =   &   w^\top  A  (x-\mu)  \mathds{1}_{u^\top  A (x-\mu)  \ge c  }  \\
		& = &  (A^\top w)^\top  x  \  \mathds{1}_{(A^Tu)^\top x  \ge c   + u^\top A \mu }   +  (A^\top w)^\top \ \mu   \ \mathds{1}_{(A^Tu)^\top x  \ge c   + u^\top A \mu }  \\
		& = &  k^\top  x \  \mathds{1}_{s^\top x  \ge c'}   +  c''  \  \mathds{1}_{s^\top x  \ge c'} 
	\end{eqnarray*}
	with $ k =  A^\top w$, $s =A^Tu$,  $  c'=c   + u^\top A \mu$,  and $c'' =    (A^\top w)^\top \ \mu  $, which implies that 
	\begin{eqnarray*}
		\Big \{ f_{\theta,\eta} \mid  \  \   (\theta, \eta) \in \Theta  \times H  \Big \}  \subseteq  \mathcal G + \mathcal G
	\end{eqnarray*}
	with $\mathcal G$ the class defined in (\ref{ClassG}).  As shown in the proof of Theorem \ref{Weakconv}, the class $\mathcal G$ is a VC-class and hence in particular Euclidean, meaning that the covering numbers grow polynomially as the radius decreases, see \citet[p.~64]{wellner2005empirical}.  It follows from the preservation of the Euclidean property under sums \citep[Proposition 8.5]{wellner2005empirical} that the class in (\ref{Donskf}) is also $P$-Donsker.  

	To show \eqref{Cond1}, it follows from law of large numbers and the continuous mapping theorem that 
	\begin{eqnarray*}
		(\widehat \Sigma_n^{-1/2}, \bar{X}_n)  \to_{\mathbb P}  (\Sigma^{-1/2}_0, \mu_0).
	\end{eqnarray*}
	Let $w  = v  - (v^\top u) u$. We write  
	\begin{eqnarray*}
		&&f_{\theta, \widehat{\eta}_n}(x)  -  f_{\theta,\eta_0}(x)  \\
		&&=   w^\top \left(\widehat \Sigma^{-1/2}_n (x-\bar{X}_n)\mathds{1}_{u^\top \widehat \Sigma^{-1/2}_n (x-\bar{X}_n) \ge c}  -  \Sigma^{-1/2}_0 (x- \mu_0)\mathds{1}_{u^\top  \Sigma^{-1/2}_0 (x-\mu_0) \ge c}  \right)\\
		&& = w^\top \left(\widehat \Sigma^{-1/2}_n (x-\bar{X}_n)  -  \Sigma^{-1/2}_0 (x- \mu_0)  \right )\mathds{1}_{u^\top \widehat \Sigma^{-1/2}_n (x-\bar{X}_n) \ge c} \\
		&&  \ +  \  w^\top  \Sigma^{-1/2}_0 (x- \mu_0)   \left(\mathds{1}_{u^\top \widehat \Sigma^{-1/2}_n (x-\bar{X}_n) \ge c}  - 
		\mathds{1}_{u^\top  \Sigma^{-1/2}_0 (x-\mu_0) \ge c}  \right)  \\
		&&  =  w^\top   \Big((\widehat \Sigma^{-1/2}_n  -  \Sigma^{-1/2}_0) x  -  (\widehat{\Sigma}^{-1/2}_n -  \Sigma^{-1/2}_0)  \bar{X}_n  \\
		&&  \  \ \   \  \  \   \  \   -  \Sigma^{-1/2}_0 (\bar{X}_n  - \mu_0) \Big)   \mathds{1}_{u^\top  \Sigma^{-1/2}_0 (x-\mu_0) \ge c}    \\
		&&   \  \ +  \  w^\top  \Sigma^{-1/2}_0 (x- \mu_0)   \left(\mathds{1}_{u^\top \widehat \Sigma^{-1/2}_n (x-\bar{X}_n) \ge c}  - 
		\mathds{1}_{u^\top  \Sigma^{-1/2}_0 (x-\mu_0) \ge c}  \right)  \\
		&& =  A_n(x)  +  B_n(x).
	\end{eqnarray*}
	Since $\Vert w \Vert \le 2$, it follows that
	\begin{eqnarray*}
		P(\|A_n\|^2)  & \le  & 8   \ ||| \widehat \Sigma^{-1/2}_n  -  \Sigma^{-1/2}_0 |||^2  \  E [ \Vert X \Vert^2]  +  8  \ ||| \widehat \Sigma^{-1/2}_n  -  \Sigma^{-1/2}_0 |||^2  \Vert \bar{X}_n \Vert^2   \\
		&&  +  \ 8  \ ||| \Sigma^{-1/2}_0 |||^2  \Vert \bar{X}_n -  \mu_0 \Vert^2  \\
		&&  \to_{\PP}  0,
	\end{eqnarray*}
	where  $||| \cdot ||| $ is the spectral norm of a symmetric $d \times d$ matrix. To handle  the term $B_n(x)$, note that
	\begin{eqnarray*}
		&&  \big \vert \mathds{1}_{u^\top \widehat \Sigma^{-1/2}_n (x-\bar{X}_n) \ge c}  -   \mathds{1}_{u^\top  \Sigma^{-1/2}_0 (x-\mu_0) \ge c}   \big \vert \\
		&&\le  \mathds{1}_{u^\top \widehat \Sigma^{-1/2}_n (x-\bar{X}_n) \ge c,  \  u^\top  \Sigma^{-1/2}_0 (x-\mu_0) < c}  + \   \mathds{1}_{u^\top \widehat \Sigma^{-1/2}_n (x-\bar{X}_n) < c, \   u^\top  \Sigma^{-1/2}_0 (x-\mu_0) \ge c}.
	\end{eqnarray*}
	Now, for any $M > 0$ we have that
	\begin{eqnarray*}
		&& \left \{u^\top \widehat \Sigma^{-1/2}_n (x-\bar{X}_n) \ge c,  \  u^\top  \Sigma^{-1/2}_0 (x-\mu_0) < c \right \}\\
		&& \subset  \left \{u^\top  \Sigma^{-1/2}_0 (x-\mu_0) \ge c  - \delta_n,  \  u^\top  \Sigma^{-1/2}_0 (x-\mu_0) < c , \  \Vert x \Vert \le M \right \}  \\
		&&   \  \  \   \cup  \  \{   \Vert x \Vert >  M  \}    
	\end{eqnarray*}
	with 
	\begin{eqnarray*}
		\vert \delta_n \vert  \le &  ||| \widehat{\Sigma}^{-1/2}_n -  \Sigma^{-1/2}_0  |||  \   M   +  ||| \widehat{\Sigma}^{-1/2}_n -  \Sigma^{-1/2}_0  ||| \  \Vert \bar{X}_n \Vert \\
		& \qquad  +   ||| \Sigma^{-1/2}_0 |||  \  \Vert \bar{X}_n -  \mu_0 \Vert  =    o_{\mathbb P}(1),
	\end{eqnarray*}
	implying that for any $\nu > 0$  and $\epsilon > 0$ there exists $n_0 \in \mathbb N$ such that for all $n \ge n_0$
	\begin{eqnarray*}
		&& \PP\left(u^\top \widehat \Sigma^{-1/2}_n (X-\bar{X}_n) \ge c,  \  u^\top  \Sigma^{-1/2}_0 (X-\mu_0) < c \right)  \\ 
		&&\le   \PP\left(u^\top  \Sigma^{-1/2}_0 (X-\mu_0) \ge c  - \nu,  \  u^\top  \Sigma^{-1/2}_0 (X-\mu_0) < c\right)  + 2\epsilon/3
	\end{eqnarray*}
	by choosing $M$ such that $\PP(\Vert X \Vert > M) \le \epsilon/3$ and $n_0$ large enough such that $\PP(\delta_n > \nu)  \le \epsilon/3$.  
	Since $u \in \mathcal{S}_{d-1}$ and $Y = \Sigma^{-1/2}_0 (X - \mu_0)$ is spherically symmetric, we have  that  $u^\top  \Sigma^{-1/2}_0 (X-\mu_0)  \stackrel{d}{=} V$ with $V$ distributed as the first coordinate of $Y$.  Thus,  it follows that 
	\begin{eqnarray*}
		&& P\left(u^\top \widehat \Sigma^{-1/2}_n (X-\bar{X}_n) \ge c,  \  u^\top  \Sigma^{-1/2}_0 (X-\mu_0) < c \right)  \\
		&& \le P\bigg ( V  \in [ c -\nu, c ) \bigg)  + 2\epsilon/3  \\
		&& = F_{V}(c)  -  F_{V}(c- \nu)   + 2 \epsilon/3\\
		&& \le \epsilon,
	\end{eqnarray*}
	for all $c \in \RR$, using (uniform) continuity of the distribution of  $V$ and taking $\nu $ to be small enough. % Indeed,  continuity of this distribution implies that the distribution function of $V$, $F_{V}$, is also continuous. Thus, the latter has to be uniformly continuous on $[\tilde{c}, \tilde{c}]$ for some $\tilde{c} > 0$, which can be chosen large enough so that $  F_{V}(c)  -F_{V}(c) \in (0, \epsilon)$ for $\vert c \vert > \tilde{c}$. 
	By the Cauchy-Schwarz inequality, it follows that 
	\begin{eqnarray*}
		P(\|B_n\|^2)  \leq 4 \,\epsilon\, E[\Vert Y\Vert^2]  
	\end{eqnarray*}
	for sufficiently large $n$. Since the second indicator in $B_n(x)$ can be handled similarly, (\ref{Cond1}) follows.  
	
\end{proof}

\begin{lemma}\label{ConvSigma}
	Suppose that $X_1, \ldots, X_n$ are i.i.d. $ \in \RR^d$ such that $\Ex[\Vert  X_1 \vert^4 ]  < \infty$, with common mean $\mu_0 $ and covariance matrix $\Sigma_0$ assumed to be positive definite. If $f: \RR^d \to \RR$ is such that $P f^2 < \infty$ and $E[X^2 \, f^2(X)]< \infty$, then  as $n \to \infty$
	\begin{eqnarray*}
		\sqrt{n} \left( 
		\begin{array}{c}
			\widehat \Sigma^{-1}_n - \Sigma^{-1}_0 \\
			\widehat \Sigma^{-1/2}_n   -  \Sigma^{-1/2}_0\\
			\hspace{0.2cm}  \bar{X}_n  - \mu_0\\
			\hspace{0.2cm}  \PP_n(f)  - P f
		\end{array}
		\right)  &\Rightarrow& 
		\left(
		\begin{array}{c}
			-\Sigma^{-1}_0 \mathbb{K}  \Sigma^{-1}_0 \\
			-\int_0^\infty e^{-t \Sigma^{-1/2}_0}   \Sigma^{-1}_0 \mathbb{K}  \Sigma^{-1}_0 e^{-t \Sigma^{-1/2}_0}  \, \dd t  \stackrel{d}{=} \mathbb S \\
			\mathbb D\\
			\mathbb{G}(f)
		\end{array}
		\right) %\\
		%& = & 
		%\comf{\left(
		%\begin{array}{lll}
		%-\Sigma^{-1}_0 \mathbb{K}  \Sigma^{-1}_0 \\
		%\  \  \ \ \  \  \mathbb{S} \\
		%\ \ \  \Sigma^{1/2} _0 \mathbb{Z}\\
		%\ \  \  \mathbb{G}(f)
		%\end{array}
		%\right).}
		%& \stackrel{d}{=}&
		% \left(
		%\begin{array}{lll}
		%\mathbb{Q} \\
		%\mathbb{S} \\
		%\Sigma^{1/2} _0 \mathbb{Z}
		%\end{array}
		%\right)
	\end{eqnarray*}
	where $\mathbb K$  and  $\mathbb D$ are a centered  Gaussian  $d \times d$ matrix and  $d$-dimensional  vector  respectively such that  for any vector $a \in \RR^d$ the covariance matrix of the $(2d+1)$-dimensional centered Gaussian vector  $(\mathbb K a, \mathbb D, \mathbb{G} (f))^\top$ is the $(2 d+1)\times (2d+1)$ matrix $\Gamma(a)$ given by 
	\begin{eqnarray*}
		\Gamma(a) = \left(
		\begin{array}{ccc}
			\Gamma_{11}(a)   &        \Gamma_{12}(a) &  \Gamma_{13}(a,f)  \\
			\Gamma_{12}(a)^\top  &   \Sigma_0 & \Gamma_{23}(f) \\
			\Gamma_{13}(a,f)^\top & \Gamma_{23}(f)^\top & P f^2 - (Pf)^2
		\end{array}
		\right)
	\end{eqnarray*}
	with
	\begin{eqnarray*}
		\Gamma_{11}(a)  &= & E\left[  \left((X - \mu_0)  (X- \mu_0)^\top -\Sigma_0\right) a  a^\top   \left((X - \mu_0)  (X- \mu_0)^\top -\Sigma_0\right) \right], \\
		%&&   \  \   \textrm{and}\\
		\Gamma_{12}(a)  &=  &    E\left[  \left((X - \mu_0)  (X- \mu_0)^\top -\Sigma_0\right) a   (X  - \mu_0)^\top \right],\\
		\Gamma_{13}(a,f)  &=  &    E\left[  \left((X - \mu_0)  (X- \mu_0)^\top -\Sigma_0\right) a  \, (f(X)  - P f) \right],\\
		\Gamma_{23}(f)  &=  &    E\left[  (X  - \mu_0)\, (f(X)  - P f) \right].
	\end{eqnarray*}
	%and $\mathbb{Z}  =  \Sigma^{-1/2}_0 \mathbb{D} \sim \mathcal{N}(0, I_d)$.
	
	%\begin{eqnarray*}
	%\sqrt n a^\top (\widehat \Sigma^{-1}_n  - \Sigma^{-1}_0) a \Rightarrow    \sigma(\Sigma^{-1}_0 a )    Z 
	%\end{eqnarray*}   
	%with  $Z \sim \mathcal{N}(0,1)$  and for any vector $b \in \mathbb R^d$
	%\begin{eqnarray*}
	%\sigma^2(b)  =  E\Big((b^\top (X_1 - \mu_0))^4\Big)  -  E \Big((b^\top (X_1 - \mu_0))^2\Big)^2.
	%\end{eqnarray*}
\end{lemma} 

\medskip

\begin{proof}[Proof of Lemma \ref{ConvSigma}] 
	% We start with establishing joint asymptotic normality for $\sqrt n (\widehat \Sigma_n   - \Sigma_0)$ and $\sqrt n (\bar{X}_n - \mu_0)$.  
	It is well-known and easy to see that $\widehat \Sigma_n$ is asymptotically equivalent to 
	$n^{-1}   \sum_{i=1}^n (X_i  - \mu_0) (X_i - \mu_0)^\top$. Thus, from the central limit theorem, we have
	\begin{eqnarray*}
		\sqrt{n} \left(
		\begin{array}{c}
			\frac{1}{n}   \sum_{i=1}^n (X_i  - \mu_0) (X_i - \mu_0)^\top   -  \Sigma_0\\
			\bar{X}_n - \mu_0  \\
			\PP_n(f)  - P f
		\end{array}
		\right)  \Rightarrow 
		\left(
		\begin{array}{c}
			\mathbb{K} \\
			\mathbb{D} \\
			\mathbb{G}(f)
		\end{array}
		\right)
	\end{eqnarray*}
	where the covariance matrix of the centered Gaussian vector $(\mathbb K a, \mathbb D)^\top \in \RR^{2d}$ is $\Gamma$ given above. 
	%&& =  \sqrt n  \left( n^{-1}  \sum_{i=1}^n ((X_i -\mu_0)^\top a)^2  -  a^\top \Sigma_0 a \right) \Rightarrow   \sigma(a)  Z 
	%\end{eqnarray*}
	%where $Z \sim \mathcal{N}(0,1)$ and 
	%\begin{eqnarray*}
	%\sigma^2(b) = Var((b^\top (X_1 - \mu_0) )^2 =  E\Big((b^\top (X_1 - \mu_0))^4\Big)  -  E \Big((b^\top (X_1 - \mu_0))^2\Big)^2. 
	%\end{eqnarray*}
	
	Now, the operator $\Sigma \mapsto \Sigma^{-1}$  defined on the sub-space of invertible matrices in $\RR^{d \times d}$ is differentiable  at any invertible matrix $A$ with gradient  $H \mapsto - A^{-1}  H A^{-1}$ , where $H \in \RR^{d \times d}$.  Also, it is known that the operator $\Sigma \mapsto \Sigma^{1/2}$, when defined on the space of positive definite matrices, is differentiable with gradient at $A$ given by 
	\begin{eqnarray*}
		H \mapsto  \int_0^\infty e^{-t A^{1/2}}  H e^{-t A^{1/2}}  dt
	\end{eqnarray*}
	which can be shown to be the solution, $S$,  of the Sylvester equation $A^{1/2}  S +  S A^{1/2}  =  H$.

	Using the delta-method and the asymptotic equivalence mentioned above,  it follows that 
	\begin{eqnarray*}
		\sqrt{n} \left(
		\begin{array}{c}
			\widehat \Sigma^{-1}_n   -  \Sigma^{-1}_0\\
			\widehat \Sigma^{-1/2}_n   -  \Sigma^{-1/2}_0\\
			\bar{X}_n - \mu_0 \\
			\PP_n(f)  - P f 
		\end{array}
		\right)  \Rightarrow 
		\left(
		\begin{array}{c}
			-\Sigma^{-1}_0\mathbb{K} \Sigma^{-1}_0\\
			-\int_0^\infty e^{-t \Sigma^{-1/2}_0}   \Sigma^{-1}_0 \mathbb{K}  \Sigma^{-1}_0 e^{-t \Sigma^{-1/2}_0}  \dd t  \\
			\mathbb{D}\\
			\mathbb{G}(f)
		\end{array}
		\right)
	\end{eqnarray*}
	which completes the proof.
\end{proof}

\bigskip

%%%%%%%%%%%%%%%%%%%%%%%%End of document%%%%%%%%%%%%%%%%%%%%%%%

\begin{proof}[Proof of Corollary \ref{Weakconvresc}]
	Since $ P   f_{\theta,\eta_0} =0$ for any $\theta \in \Theta$, 
	\begin{eqnarray*}
		\sqrt n \mathbb P_n  f_{\theta,\widehat{\eta}_n }   &=   &   \sqrt n \left(\mathbb P_n  f_{\theta,\widehat{\eta}_n }   -  P   f_{\theta,\eta_0}   \right )    \\
		& =   &    \mathbb{G}_n  (f_{\theta,\widehat{\eta}_n } -  f_{\theta,\eta_0})   +  \mathbb{G}_n f_{\theta, \eta_0} +  \sqrt n \,  P  f_{\theta, \widehat \eta_n}.
	\end{eqnarray*}
	In view of Theorem \ref{Th: empestim}, the first term is a process which converges weakly to $0$. By Theorem \ref{Weakconv}, we know that $\mathbb{G}_n f_{\theta, \eta_0} $  converges weakly to $\mathbb G$. Hence, it remains to find the weak limit of the third term (jointly with $\mathbb G$).  To this end, first fix  $\theta = (u,v, c) \in \Theta$. As shown in the proof of Theorem \ref{Th: empestim}, setting  $w= v - (v^\top u) u$  we have that
	\begin{eqnarray*}
		\sqrt n  P  f_{\theta, \widehat \eta_n}   & =   &    \sqrt n  P  (f_{\theta, \widehat \eta_n}   - f_{\theta,\eta_0}) \\  
		& = &   I_n  +   II_n 
	\end{eqnarray*}
	where 
	\begin{eqnarray*}
		I_n  &= &  w^\top  \widehat \Sigma^{-1/2}_n \sqrt n (\mu_0 - \bar{X}_n)  P(u^\top  Y > c)  +  \\
		&&  \ +   w^\top \sqrt n (\widehat{\Sigma}_n^{-1/2}  - \Sigma^{-1/2}_0)  \Sigma^{1/2}_0  E[Y \mathds{1}_{u^\top Y > c}]  
	\end{eqnarray*}
	and 
	\begin{eqnarray*}
		&& II_n\\
		& & =     w^\top \sqrt n (\widehat \Sigma^{-1/2}_n  - \Sigma^{-1/2}_0) \\
		&& \  \ \  \    \times \  E \left[(X-\mu_0)  \big(\mathds{1}_{u^\top \widehat \Sigma^{-1/2}_n (X- \bar{X}_n) > c}  -  \mathds{1}_{u^\top  \Sigma^{-1/2}_0 (X- \mu_0) > c }\big)\right]\\
		& &  \  \  + \   w^\top  \sqrt n  \Sigma^{-1/2}_0 E \left[(X-\mu_0)  \big(\mathds{1}_{u^\top \widehat \Sigma^{-1/2}_n (X- \bar{X}_n) > c}  -  \mathds{1}_{u^\top  \Sigma^{-1/2}_0 (X- \mu_0) > c }\big)\right]  \\
		& &   \  \  + \  w^\top   \widehat \Sigma^{-1/2}_n   \sqrt n  (\mu_0 -  \bar{X}_n) \\
		&&  \  \  \ \ \times \   \left( P(u^\top \widehat \Sigma^{-1/2}_n (X- \bar{X}_n) > c)  -   P(u^\top  \Sigma^{-1/2}_0 (X- \mu_0) > c )  \right).
	\end{eqnarray*}
	It follows from  Lemma \ref{ConvSigma} and the continuous mapping theorem that 
	\begin{eqnarray*}
		I_n \Rightarrow -w^T \mathbb{Z} P(u^T Y > c)  + w^T \mathbb{S} \Sigma^{1/2}_0  E[Y \mathds{1}_{u^T Y > c} ] 
	\end{eqnarray*}
	as $n \to \infty$, where $\mathbb Z = \Sigma_0^{-1/2} \mathbb{D}$, and $\mathbb{D}$ and $\mathbb{S}$ are as in Lemma \ref{ConvSigma}. Using spherical symmetry of $Y$, the weak convergence above can be given by the equivalent form 
	\begin{eqnarray}\label{ConvIn}
	I_n \to _d -w^T \mathbb{Z} P(V > c)  + w^T \mathbb{S} \Sigma^{1/2}_0  u E[V \mathds{1}_{V > c} ] 
	\end{eqnarray}
	where $V$ is distributed for example as the first coordinate of $Y$.
	%{\color{red} I can't follow the line of argument following below. The terms in $I_n$ can be dealt with directly by using Lemma \ref{ConvSigma} and are not asymptotically negligible. The first and the third term in $II_n$ are negligible. You do computations on the second term in $II_n$, what is the result? From your computations it seems that this term should contribute to the asymptotics, but I cannot see how this is included in your result}.

	% Let $Z  \sim \mathcal{N}(0, I_d)$ be defined on the same probability space as $\mathbb G_n$. Then, it follows from the Central Limit Theorem and Lemma \ref{ConvSigma} below that
	%\begin{eqnarray*}
	%I_n \Rightarrow    w^\top   Z  P(u^\top  Y > c).
	%\end{eqnarray*}
	Note first that both the first and last term in $II_n$ converge to $0$ by Lemma \ref{ConvSigma}, the Central Limit Theorem and the arguments used below showing that $P(u^\top \widehat \Sigma^{-1/2}_n (X- \bar{X}_n) > c)  -   P(u^\top \widehat \Sigma^{-1/2}_0 (X- \mu_0) > c )  = O_p(n^{-1/2})$ (we can also use for this part arguments similar to those used in the proof of Theorem \ref{Th: empestim}).   For the middle term, $II_{n, 2}$ say,  we can write 
	\begin{eqnarray*}
		&& E \left[(X-\mu_0)  \big(\mathds{1}_{u^\top \widehat \Sigma^{-1/2}_n (X- \bar{X}_n) > c}  -  \mathds{1}_{u^\top  \Sigma^{-1/2}_0 (X- \mu_0) > c }\big)\right]   \\
		&& =  \Sigma_0^{1/2}E \left[Y \mathds{1}_{u^\top \widehat \Sigma^{-1/2}_n (\Sigma^{1/2} _0  Y  + \mu_0- \bar{X}_n) > c} \right] - \Sigma^{1/2}_0  E\left( Y \mathds{1}_{u^\top  Y > c }\right)\\
		&&  =   \Sigma_0^{1/2}  E\left[  Y  (\mathds{1}_{u^\top_n  Y > c_n}  -    \mathds{1}_{u^\top  Y > c}) \right]  \\
		&& = \Sigma_0^{1/2} \left\{  E\left[  Y \mathds{1}_{u^\top_n  Y > c_n} \right]  -   E\left[  Y \mathds{1}_{u  Y > c} \right] \right\}
	\end{eqnarray*}
	with $u_n=  \Sigma^{1/2}_0 \widehat \Sigma^{-1/2}_n    u $,  $c_n =  c +  u^\top \widehat \Sigma^{-1/2}_n(\bar{X}_n - \mu_0)$ and $Y = \Sigma^{-1/2}_0(X-\mu_0)$. It follows that 
	\begin{eqnarray}\label{MiddleTerm}
	II_{n,2}  =  w^T  \sqrt n \left\{  E\left[  Y \mathds{1}_{u^\top_n  Y > c_n} \right]  -   E\left[  Y \mathds{1}_{u  Y > c} \right] \right\}.
	\end{eqnarray}
	Since $Y$ is spherically symmetric, we have that  
	\begin{eqnarray*}
		E\left[  Y \mathds{1}_{u^\top_n  Y > c_n} \right] & =   &  u_n E\left[u^\top_n  Y \mathds{1}_{u^\top_n  Y > c_n}\right]  \\
		& = &  u_n \Vert u_n \Vert   E\left[\tilde{u}^\top_n  Y \mathds{1}_{\tilde{u}^\top_n  Y > c_n  \Vert u_n \Vert^{-1}}\right], \  \  \textrm{with $\tilde{u}_n =  u_n / \Vert u_n \Vert $}\\
		& = &    u_n \Vert u_n \Vert  E\left[V  \mathds{1}_{V > c_n  \Vert u_n \Vert^{-1}}\right]
	\end{eqnarray*}
	where $V$  denotes again the first component of the vector spherically symmetric vector $Y$.  Similar arguments and the fact that $\Vert u \Vert =1$  imply that 
	\begin{eqnarray}\label{Exp}
	E\left[  Y \mathds{1}_{u^\top_n  Y > c_n} \right]  -   E\left[  Y \mathds{1}_{u^\top  Y > c} \right] 
	&=   &  u_n \Vert u_n \Vert E\left[V  \mathds{1}_{V > c_n  \Vert u_n \Vert^{-1}}\right]  -  u E\left[ V  \mathds{1}_{V > c}\right] \notag \\
	& = & u_n \Vert u_n \Vert\left(  E\left[V  \mathds{1}_{V > c_n  \Vert u_n \Vert^{-1}}\right]   -   E\left[ V  \mathds{1}_{V > c}\right]   \right) \notag \\
	&&  +  \  \  (u_n \Vert u_n \Vert - u) E\left[ V  \mathds{1}_{V > c}\right].
	%& = &   (\Vert u_n \Vert  - 1)  \Ex[V \mathds{1}_{V > c}] \\
	%&& \ \   +  \  \Vert u_n \Vert  \Ex[ V  (\mathds{1}_{V >  c_n  \Vert u_n \Vert^{-1}}\  -  \mathds{1}_{V > c} )].
	\end{eqnarray}
	Let $h_n =  c_n \Vert u_n \Vert^{-1}  - c $, and $g$ denote the density of $V$ with respect to Lebesgue measure on $\RR$.  Then
	\begin{eqnarray*}
		\sqrt n\  E[ V  (\mathds{1}_{V >  c_n  \Vert u_n \Vert^{-1}}\  -  \mathds{1}_{V > c} )] & = & - \sqrt n  \int_{c}^ {c_n  \Vert u_n \Vert^{-1}}  v g(v)  \dd y   \\
		&=&  -\sqrt n h_n ( c g(c)  +  o(1))
	\end{eqnarray*}
	where $o(1)$ does not depend on $c$ since $v \mapsto v g(v) $ is assumed to be uniformly bounded. Now we compute 
	\begin{eqnarray*}
		\sqrt n h_n &  =  &  \sqrt n \left(  \frac{c+  u^\top \widehat \Sigma^{-1/2}_n (\bar{X}_n - \mu_0)  }{\Vert u_n \Vert}  -  c \right)\\
		& = &  \frac{1}{\Vert u_n \Vert}  \sqrt n \left (c  +  u^\top  \widehat \Sigma^{-1/2}_n (\bar{X}_n - \mu_0)  -  c \Vert u_n \Vert \right)  \\
		& = &  -\frac{c}{\Vert u_n \Vert}  \sqrt n (\Vert u_n \Vert - 1)  +  \frac{u^\top}{\Vert u_n \Vert}   \sqrt n  \widehat \Sigma^{-1/2}_n (\bar{X}_n - \mu_0) \\
		& = &  -\frac{c}{\Vert u_n \Vert (1 + \Vert u_n \Vert)}  \sqrt n (\Vert u_n \Vert^2 - 1)   +  \frac{u^\top}{\Vert u_n \Vert}   \sqrt n  \widehat \Sigma^{-1/2}_n (\bar{X}_n - \mu_0).
	\end{eqnarray*}
	Furthermore,
	\begin{eqnarray*}
		\sqrt n (\Vert u_n \Vert^2 - 1)  =  -u^\top \widehat{\Sigma}^{-1/2}_n  \sqrt n  (\widehat{\Sigma}_n - \Sigma_0)  \widehat{\Sigma}^{-1/2}_n  u 
	\end{eqnarray*}
	so that
	\begin{eqnarray*}
		\sqrt n h_n  & =   & \frac{c}{\Vert u_n \Vert (1 + \Vert u_n \Vert)}   u^\top \widehat{\Sigma}^{-1/2}_n  \sqrt n  (\widehat{\Sigma}_n - \Sigma_0)  \widehat{\Sigma}^{-1/2}_n  u   \\
		&&  +  \  \ \frac{u^\top}{\Vert u_n \Vert}   \sqrt n  \widehat \Sigma^{-1/2}_n (\bar{X}_n - \mu_0).
	\end{eqnarray*}
	By Lemma \ref{ConvSigma}, $\sqrt n h_n$ admits a weak limit as $n \to \infty$. Also, it follows from the same lemma that $u_n \Vert u_n \Vert $ converges to $u$ in probability as $n \to \infty$.  Since $w^T u =0$, this implies that 
	\begin{eqnarray*}
		w^T u_n \Vert u_n \Vert \sqrt n  E[ V  (\mathds{1}_{V >  c_n  \Vert u_n \Vert^{-1}}\  -  \mathds{1}_{V > c} )]  \to 0
	\end{eqnarray*}
	in probability as $n \to \infty$. Now, we get to the second term in (\ref{Exp}).   We have that 
	\begin{eqnarray*}
		\sqrt n (u_n \Vert u_n \Vert -  u) & =&  \sqrt n (u_n -  u)  + \sqrt n  u_n ( \Vert u_n \Vert -  1)  \\ 
		& = &  \sqrt n (\Sigma^{1/2}_0 \widehat \Sigma^{-1/2}_n   u - u)  +  \frac{u_n}{\Vert u_n \Vert + 1}  \sqrt n (\Vert u_n \Vert^2  - 1)   \\
		& = &  \sqrt n  \Sigma^{1/2}_0 (\widehat \Sigma^{-1/2}_n   - \Sigma^{-1/2}_0) u  +  \frac{u_n}{\Vert u_n \Vert + 1}  \sqrt n (\Vert u_n \Vert^2  - 1)     \\
		& = &  \sqrt n  \Sigma^{1/2}_0 (\widehat \Sigma^{-1/2}_n   - \Sigma^{-1/2}_0) u   \\
		&& \  \ -  \frac{u_n}{\Vert u_n \Vert + 1}   u^\top \widehat{\Sigma}^{-1/2}_n  \sqrt n  (\widehat{\Sigma}_n - \Sigma_0)  \widehat{\Sigma}^{-1/2}_n  u. 
	\end{eqnarray*}
	Thus, only the first term in the preceding display will contribute using again that $w^T u_n/ (\Vert u_n \Vert +1 )  \to_p 0$ as $n \to \infty$. Using the expression of $II_{n, 2}$ in (\ref{MiddleTerm}) it follows from our calculations above that 
	\begin{eqnarray*}
		II_{n, 2}  \Rightarrow w^\top \Sigma^{1/2}_0  \mathbb{S} u E[V \mathds{1}_{V > c}],
	\end{eqnarray*}
	where we recall that $V$ distributed as the first coordinate of $Y$.  Putting this together with the weak convergence in (\ref{ConvIn}) it follows that  
	\begin{eqnarray*}
		\sqrt n P f_{\theta, \widehat \eta_n}  \Rightarrow -w^T \mathbb Z P(V > c)  + 2 w^T \mathbb S\Sigma^{1/2}_0  u E[ V \mathds{1}_{ V > c} ].
	\end{eqnarray*}

	Putting all pieces together, using the fact that $w^\top u =0$ and the joint weak convergence established in Lemma \ref{ConvSigma}, the claimed weak convergence 
	\begin{eqnarray*}
		\sqrt n \mathbb P_{\theta, \widehat \eta_n}  \Rightarrow  \mathbb L(\theta)
	\end{eqnarray*}
	follows for that chosen $\theta$. To show that this weak convergence holds for the whole process converges in $\ell^\infty(\Theta)$, it is enough to show that this process is tight and apply the Prohorov's Theorem, see e.g. Theorem 1.3.9 in \cite{vdvwell}. Indeed, uniqueness of the weak limit for a fixed $\theta$ will imply that the process converges weakly to $\mathbb L$.  Following the calculations detailed above, it is easy to see that for any $\epsilon > 0$, there exists a constant $C_\epsilon > 0$  such that 
	\begin{eqnarray*}
		P(\Vert \sqrt n \mathbb P f_{\cdot , \widehat{\eta}_n}\Vert_\Theta \le C_\epsilon)  \ge 1-\epsilon.
	\end{eqnarray*}
	By considering the compact set $K_\epsilon =  \{ \psi \in \ell^\infty(\Theta): \Vert \psi \Vert_\infty \le C_\epsilon \}$, we see that  the preceding inequality gives tightness.  The joint Gaussianity of $L(\theta)$ and $\mathbb G$ follows from  of Lemma \ref{ConvSigma}. This completes the proof. 
\end{proof}

\bigskip

\begin{proof}[Proof of Theorem \ref{TheoModifBoot}]
	The proof proceeds in the following steps. 
	
	\emph{Step 1.:} For $i =1, \ldots, n$, let us define 
	\begin{eqnarray*}
		Z_{ni}  = n^{-1/2} \delta_{\widehat X^*_i}  =  n^{-1/2}  \delta_{\widehat R^*_i W_i} 
	\end{eqnarray*}
	where we recall that $\widehat R^*_1, \ldots, \widehat R^*_n$ are i.i.d. random variables from the empirical distribution of the norms of the standardized observed vectors, that is, 
	\begin{eqnarray}\label{hatR}
	\widehat R_1 =  \Vert \widehat \Sigma^{-1/2}_n (X_1 -  \bar{X}_n)  \Vert, \ldots, \widehat R_n = \Vert \widehat \Sigma^{-1/2}_n (X_n -  \bar{X}_n)\Vert, 
	\end{eqnarray}
	and $W_1, \ldots, W_n$ are independently and uniformly sampled from the unit sphere $\mathcal{S}_{d-1}$, and such that they are also independent of $\widehat R_1, \ldots, \widehat R_n$.  
	Then the process $\{ \sum_{i=1}^n Z_{ni}(f), f \in \mathcal{F}_s\}$ converges weakly to $\mathbb G f, f \in \mathcal F_s$, where 
	$\mathcal F_s$ is as in Theorem \ref{CondEmpProc}, and the process $\mathbb G$ is defined in Theorem \ref{CondEmpProc}.   
	We apply this with $\mathcal F_s = \mathcal F$, where $\mathcal F$ is as in \eqref{classF}. 
	%the class defined in (2.3), which we know that it is   is a VC-subgraph such that $\mathbb P f = 0, \ \forall \ f \in \mathbb F$ for a probability measure $\mathbb P$ associated with a elliptically symmetric distribution, and that it admits 
	%\begin{eqnarray*}
	% F(x)  = 2 \Vert x \Vert 
	%\end{eqnarray*}
	%as an envelope function $F$.  
	
	\bigskip
	
	\emph{Step 2.:} In this step we derive the limit distribution of the bootstrapped matrix $\widehat \Sigma^*_n$ given $X_1, X_2, \ldots,$, and show that conditional on $X_1, X_2, \ldots,$, almost surely,
	\begin{eqnarray}\label{LimitSigmastar}
	\sqrt n (\widehat \Sigma^*_n  - \mathcal{I}_d )  \Rightarrow  \mathbb K_0,
	\end{eqnarray}
	with $\mathcal{I}_d$ the $d \times d$ identity matrix and $\mathbb K_0$ a $d \times d$ matrix with centered Gaussian components such that for any vector $a \in \mathbb R^d$, the covariance matrix of $\mathbb K_0 a$ is given by 
	\begin{eqnarray*}
		\Gamma_0(a) =  \mathbb E\left[  (Y Y^\top   - \mathcal{I}_d) aa^\top  (YY^\top  - \mathcal{I}_d) \right].
	\end{eqnarray*}
	
	\bigskip
	
	\emph{Step 3.:} % This implies that $(\widehat \Sigma^*_n)^{-1/2}$ is also $\sqrt n$-consistent. The goal is to  link this limit to that of $\widehat \Sigma^{-1}_n$.  
	To conclude, recall that the drift, $\mathbb L(\theta)$,  in the weak limit of Corollary \ref{Weakconvresc} is given by
	\begin{eqnarray}\label{eq:driftagain}
	\mathbb L(\theta) =  - (w^\top  \Sigma^{-1/2}_0 \mathbb D) \mathbb P(V > c)  +   w^\top (\mathbb S \Sigma^{1/2}_0  +  \Sigma^{1/2}_0  \mathbb S) u \mathbb E[ V \mathds{1}_{V > c}],
	\end{eqnarray}
	where $\theta = (u, v, c)  \in \mathcal{S}_{d-1} \times \mathcal{S}_{d-1} \times \mathbb R$, $w =  v - (v^\top u) u$,   $V$ is any component of $Y = \Sigma^{-1/2}_0  (X- \mu_0)$, and $\mathbb D$ and $\mathbb S$ are centered Gaussian vector in $\mathbb R^d$ and matrix in $\mathbb R^d \times \mathbb R^d$ such that 
	\begin{eqnarray*}
		\sqrt n \left(
		\begin{array}{c}
			\bar X_n - \mu_0 \\
			\ \ \ \widehat \Sigma^{-1/2}_n  - \Sigma^{-1/2}_0
		\end{array}
		\right)  \Rightarrow
		\left(
		\begin{array}{c}
			\mathbb D \\
			\mathbb S
		\end{array}
		\right).
	\end{eqnarray*}  
	Following the same steps of the proof of Corollary \ref{Weakconvresc} while replacing $\widehat \eta_n= (\widehat \Sigma_n^{-1/2}, \bar{X}_n)$ with 
	%\begin{eqnarray*}
	$\widehat \eta^*_n  = ((\widehat \Sigma^*_n)^{-1/2}, \bar{X}^*_n)$,
	%\end{eqnarray*}
	the associated drift has the asymptotic distribution 
	\begin{eqnarray}\label{Drift0}
	- (w^\top \mathbb D_0) \mathbb P(V > c)  +   2 w^\top  \mathbb S_0 u \mathbb E[ V \mathds{1}_{V > c}],
	\end{eqnarray}
	where
	$\mathcal D_0  \sim \mathcal{N}(0, \mathcal{I}_d)$ and $\mathbb S_0$ is the centered Gaussian $d \times d$ matrix such that
	\begin{eqnarray*}
		\sqrt n ((\widehat \Sigma^*_n)^{-1/2}  -\mathcal{I}_d)  \Rightarrow  \mathbb S_0,
	\end{eqnarray*}
	the existence of which follows from Step 2.
	Hence, in order to show that the modified bootstrap is consistent, we need to show that the drift $\LL(\theta)$ in \eqref{eq:driftagain} is also given by the expression in \eqref{Drift0}, for which we provide the details below. This finishes the proof of the theorem. 
	\emph {Details for Step 1.}\quad Since the proof of this weak convergence should go along the same lines of the proof of Theorem \ref{CondEmpProc}, one needs to check first that the three conditions of Theorem 2.11.1 in \cite{vdvwell}.  Here, we show that the third condition is indeed satisfied as the first two ones involve very similar arguments as in the proof of Theorem \ref{CondEmpProc}. For this third condition, we need to show that 
	\begin{eqnarray*}
		\lim_{n \to \infty} \text{cov}_{\widehat{L}^*_n \otimes \mathcal U}\bigg( \sum_{i=1}^n Z_{ni}(f),  \sum_{i=1}^n Z_{ni}(g) \bigg)   =   \mathbb E_{L_0 \otimes \mathcal U} \big[ f(W R)  g(WR  \big]
	\end{eqnarray*}
	where $\widehat{L}^*_n$ is the empirical distribution $1/n \sum_{i=1}^n \delta_{\widehat R_i}$, and $\mathcal U$ denotes again the uniform distribution on the sphere $\mathcal S_{d}$.   By definition of $Z_{ni}$ we have that
	\begin{eqnarray*}
		\text{cov}_{\widehat{L}^*_n \otimes \mathcal U}\bigg( \sum_{i=1}^n Z_{ni}(f),  \sum_{i=1}^n Z_{ni}(g) \bigg) & = &  n \  \text{cov}_{\widehat{L}^*_n \otimes \mathcal U} (Z_{n1}(f), Z_{n1}(g) ) \\
		& = &  \mathbb E_{\widehat{L}^*_n \otimes \mathcal U} \big[f(\widehat R^*  W) g(\widehat R^* W)   \big ] \\
		& = &  \frac{1}{n}  \sum_{i=1}^n \mathbb E_{\mathcal U} \big[f(\widehat R_i W)  g(\widehat R_i W) \big] 
	\end{eqnarray*}
	where  for $f(x)  = f_{u, v, c}(x)  =  (v -  (v^\top  u) u)^\top  x \mathds{1}_{u^\top  x \ge c}$, and $g(x)  = g_{u', v', c'}(x)  =  (v' -  (v'^\top  u') u')^\top  x \mathds{1}_{u'^\top  x \ge c'} $  for  $(u, v), (u', v')\in \mathcal{S}^2_d$ and $c, c' \in \mathbb R$ we have that
	\begin{eqnarray*}
		f(\widehat R_i W)   =  \Vert \widehat \Sigma^{-1/2}_n (X_i-  \bar{X}_n)  \Vert  (v -  (v^\top  u) u)^\top  W \mathds{1}_{  \Vert \widehat \Sigma^{-1/2}_n (X_i-  \bar{X}_n)  \Vert u^\top   W\ge c}
	\end{eqnarray*}
	and 
	\begin{eqnarray*}
		g(\widehat R_i W)   =  \Vert \widehat \Sigma^{-1/2}_n (X_i-  \bar{X}_n)  \Vert  (v' -  (v'^\top  u') u')^\top  W \mathds{1}_{  \Vert \widehat \Sigma^{-1/2}_n (X_i-  \bar{X}_n)  \Vert u'^\top   W\ge c'}.
	\end{eqnarray*}
	Since we wish to replace $\widehat R_i$  by $R_i =  \Vert \Sigma_0^{-1/2} (X_i - \mu_0) \Vert$ for $i =1, \ldots, n$, we can write that 
	\begin{eqnarray}\label{Breakdown}
	\frac{1}{n}  \sum_{i=1}^n \mathbb E_{\mathcal U} \big[f(\widehat R_i W)  g(\widehat R_i W) \big] 
	& = &  \frac{1}{n}  \sum_{i=1}^n \mathbb E_{\mathcal U} \big[f(R_i W)  g(R_i W) \big] \nonumber \\
	&&  +  \frac{1}{n}  \sum_{i=1}^n \mathbb E_{\mathcal U} \big[(f(\widehat R_i W) - f(R_i W))  g(\widehat R_i W) \big]  \nonumber \\
	&& +  \frac{1}{n}  \sum_{i=1}^n \mathbb E_{\mathcal U} \big[(f(R_i W))  (g(\widehat R_i W)-g(R_i W)) \big] \nonumber \\
	&& =  \frac{1}{n}  \sum_{i=1}^n \mathbb E_{\mathcal U} \big[f(R_i W)  g(R_i W) \big]  +  A_n + B_n. \nonumber\\
	&&
	\end{eqnarray}
	Now for $i=1, \ldots, n$ we have that 
	\begin{eqnarray*}
		&& \big \vert f(\widehat R_i W) - f(R_i W) \big \vert  \\
		&& = \bigg \vert \Vert \widehat \Sigma^{-1/2}_n (X_i-  \bar{X}_n)  \Vert  (v -  (v^\top  u) u)^\top  W \mathds{1}_{  \Vert \widehat \Sigma^{-1/2}_n (X_i-  \bar{X}_n)  \Vert u^\top   W\ge c}\\
		&&  \hspace{0.5cm}  -    \Vert \Sigma^{-1/2}_0 (X_i-  \mu_0)  \Vert  (v -  (v^\top  u) u)^\top  W \mathds{1}_{  \Vert \Sigma^{-1/2}_0 (X_i-  \mu_0)  \Vert u^\top   W\ge c} \bigg\vert \\
		&& \le \bigg \vert \bigg( \Vert \widehat \Sigma^{-1/2}_n (X_i-  \bar{X}_n) \Vert  -   \Vert  \Sigma^{-1/2}_0 (X_i-  \mu_0)  \Vert  \bigg) \\
		&& \qquad \qquad \cdot  (v -  (v^\top  u) u)^\top  W \mathds{1}_{  \Vert \widehat \Sigma^{-1/2}_n (X_i-  \bar{X}_n)  \Vert u^\top   W\ge c} \bigg \vert\\
		&&  \hspace{0.5cm} +   \  \Vert  \Sigma^{-1/2}_0 (X_i-  \mu_0)  \Vert   \times \vert (v -  (v^\top  u) u)^\top  W  \vert \\
		&& \qquad \qquad  \cdot \bigg \vert   \mathds{1}_{  \Vert \widehat \Sigma^{-1/2}_n (X_i-  \bar{X}_n)  \Vert u^\top   W\ge c}  -   \mathds{1}_{  \Vert \widehat \Sigma^{-1/2}_0 (X_i-  \mu_0)  \Vert u^\top   W\ge c}  \bigg \vert \\
		&& = I_{i,n} + II_{i,n}  
	\end{eqnarray*}
	where $I_{n, i}$ and $II_{n, ii}$ are functions of the observed data as well as $W \sim \mathcal U$.  The goal now is to find an upper bound of the expectation of these terms with respect to the distribution $\mathcal U$.   We have that 
	\begin{eqnarray*}
		&& \Vert \widehat \Sigma^{-1/2}_n (X_i-  \bar{X}_n) \Vert \\
		&& = \Vert \Sigma^{-1/2}_0(X_i  - \mu_0)  +  \big(\widehat \Sigma^{-1/2}_n  - \Sigma^{-1/2}_0\big)  (X_i-  \mu_0)   +  \widehat \Sigma^{-1/2}_n (-\bar{X}_n - \mu_0) \Vert
	\end{eqnarray*}
	and hence 
	\begin{eqnarray}\label{Ineq1}
	& &\bigg \vert \Vert \widehat \Sigma^{-1/2}_n (X_i-  \bar{X}_n) \Vert -   \Vert \Sigma^{-1/2}_0(X_i  - \mu_0)  \Vert \bigg \vert\\
	&\le   & \Vert \big(\widehat \Sigma^{-1/2}_n  - \Sigma^{-1/2}_0\big)  (X_i-  \mu_0)  \Vert +  \Vert \widehat \Sigma^{-1/2}_n (-\bar{X}_n +\mu_0) \Vert \nonumber \\
	& \le &  \Vert \widehat \Sigma^{-1/2}_n  - \Sigma^{-1/2}_0 \Vert \cdot \Vert X_i-  \mu_0 \Vert +  \Vert \widehat \Sigma^{-1/2}_n \Vert \cdot \Vert \bar{X}_n - \mu_0  \Vert, \nonumber 
	\end{eqnarray}
	where $\Vert A \Vert$ of a matrix $A$ denotes the usual operator norm defined here with respect to the Euclidean norm.  Using the Cauchy-Schwarz inequality and the fact that $\Vert  (v -  (v^\top  u) u) \Vert \le 2$ and $\Vert W \Vert =1$, the latter implies that 
	\begin{eqnarray*}
		I_{n, i} \le 2 \left(\Vert \widehat \Sigma^{-1/2}_n  - \Sigma^{-1/2}_0 \Vert \cdot \Vert X_i-  \mu_0 \Vert +  \Vert \widehat \Sigma^{-1/2}_n \Vert \cdot \Vert \bar{X}_n - \mu_0  \Vert\right)
	\end{eqnarray*}
	and hence, using again  $\Vert  (v -  (v^\top  u) u) \Vert \le 2$ and $\Vert W \Vert =1$ in addition to the inequality obtained in (\ref{Ineq1}), it follows that
	{\small 
		\begin{eqnarray*}
			& & I_{n, i} \cdot  \vert  g(\widehat R_i W) \vert\\
			&\le  &  4  \left(\Vert \widehat \Sigma^{-1/2}_n  - \Sigma^{-1/2}_0 \Vert \cdot \Vert X_i-  \mu_0 \Vert +  \Vert \widehat \Sigma^{-1/2}_n \Vert \cdot \Vert \bar{X}_n - \mu_0  \Vert\right)  \Vert \widehat \Sigma^{-1/2}_n  (X_i - \bar X_n) \Vert \\
			& \le &  4  \left(\Vert \widehat \Sigma^{-1/2}_n  - \Sigma^{-1/2}_0 \Vert \cdot \Vert X_i-  \mu_0 \Vert +  \Vert \widehat \Sigma^{-1/2}_n \Vert \cdot \Vert \bar{X}_n - \mu_0  \Vert\right)  \\
			&&   \hspace{0.3cm} \times  \left(\Vert \Sigma^{-1/2}_0(X_i  - \mu_0)  \Vert   +  \Vert \widehat \Sigma^{-1/2}_n  - \Sigma^{-1/2}_0 \Vert \cdot \Vert X_i-  \mu_0 \Vert +  \Vert \widehat \Sigma^{-1/2}_n \Vert \cdot \Vert \bar{X}_n - \mu_0  \Vert  \right) \\
			&& \le  \alpha_n  \Vert X_i - \mu_0 \Vert^2  + \beta_n   \Vert X_i - \mu_0 \Vert  + \gamma_n 
		\end{eqnarray*}
	}
	where $\alpha_n, \beta_n $ and $\gamma_n$ are $O_p(n^{-1/2})$  as a consequence of Lemma 5.2 and the weak law of large numbers. Since $\mathbb E[\Vert X- \mu_0\Vert^2]  < \infty$ is satisfied under the assumption that $\mathbb E[\Vert X\Vert^4] < \infty$ we can conclude that
	\begin{eqnarray*}
		\frac{1}{n}  \mathbb{E}_{\mathcal U} [I_{n, i} \cdot  \vert  g(\widehat R_i W) \vert]  = O_p(n^{-1/2}).
	\end{eqnarray*}
	Now, we turn to the terms $II_{n, i}, i =1, \ldots, n$. We have that 
	\begin{eqnarray}\label{Expect}
	&& \mathbb E_{\mathcal{U}}[II_{n, i}  \cdot  \vert  g(\widehat R_i W) ] \le 4  \Vert \Sigma^{-1/2}_0  (X_i - \mu_0) \Vert\\
	&& \cdot  \bigg(\mathbb P_{\mathcal{U}}\big( \Vert \widehat{\Sigma_n}^{-1/2} (X_i - \bar X_n ) \Vert u^\top  W  \le c,    \Vert \Sigma_0^{-1/2} (X_i - \mu_0 ) \Vert u^\top  W  > c\big)  \nonumber \\
	&& \quad + \mathbb P_{\mathcal U}\big( \Vert \widehat{\Sigma_n}^{-1/2} (X_i - \bar X_n ) \Vert u^\top  W  >  c,    \Vert \Sigma_0^{-1/2} (X_i - \mu_0 ) \Vert u^\top  W \le c\big)   \bigg). \nonumber 
	\end{eqnarray}
	Using again the inequality obtained in (\ref{Ineq1}), we can find $a_n = O_p(n^{-1/2})$ and $b_n = O_p(n^{-1/2})$ both  independent of $i$ such that  $\Vert \widehat{\Sigma_n}^{-1/2} (X_i - \bar X_n ) \Vert = \Vert \Sigma^{-1/2}_0 (X_i - \mu_0) \Vert  (1+ a_n) - b_n$. Hence,   
	\begin{eqnarray*}
		&&\mathbb P_{\mathcal{U}}\bigg( \Vert \widehat{\Sigma_n}^{-1/2} (X_i - \bar X_n ) \Vert u^\top  W  \le c,    \Vert \Sigma_0^{-1/2} (X_i - \mu_0 ) \Vert u^\top  W  > c\bigg) \\
		&& = \mathbb P_{\mathcal{U}}\bigg(\Sigma_0^{-1/2} (X_i - \mu_0 ) (1+ a_n) \Vert u^\top  W \in (c, c + b_n ]  \bigg) \\
		&& \le \mathbb P_{\mathcal{U}}\bigg(\Sigma_0^{-1/2} (X_i - \mu_0 )  \Vert u^\top  W \in (2c,2(c + b_n) ]  \bigg)  \\
		&& =  \frac{b_n}{\Vert \Sigma^{-1/2}_0 (X_i - \mu_0)\Vert}
	\end{eqnarray*}
	using that $u^\top W $ is uniformly distributed on $[-1,1]$. Handling the second probability in (\ref{Expect})  in a similar fashion gives  
	\begin{eqnarray*}
		\frac{1}{n} \sum_{i=1}^n \mathbb E_{\mathcal{U}}[II_{n, i}  \cdot  \vert  g(\widehat R_i W) ] = O_p(n^{-1/2})
	\end{eqnarray*}
	and we conclude that
	\begin{eqnarray*}
		A_n =  O_p(n^{-1/2}).
	\end{eqnarray*}
	Similarly, we can show that $B_n= O_p(n^{-1/2})$.  Now, by the strong law of large numbers, 
	\begin{eqnarray*}
		\frac{1}{n}  \sum_{i=1}^n \mathbb E_{\mathcal U} \big[f(R_i W)  g(R_i W) \big]  &\to_{\text{a.s.}}& \mathbb E_{L_0 \otimes \mathcal U}[f(WR) g(WR)], \\ %\textrm{ where $L_0$ is the true distribution of $R$} \\
		& =&  \mathbb E[f(Y) g(Y)] = \cov(f(Y), g(Y)), \  %\textrm{where $Y = \Sigma^{-1/2}_0 (X-\mu_0)$}. 
	\end{eqnarray*}
	where $L_0$ is the true distribution of $R$. This in turn implies that
	\begin{eqnarray*}
		\lim_{n \to \infty} \cov_{\widehat{L}^*_n \otimes \mathcal U}\bigg( \sum_{i=1}^n Z_{ni}(f),  \sum_{i=1}^n Z_{ni}(g) \bigg)  = \cov(f(Y), g(Y))
	\end{eqnarray*}
	which is exactly the covariance of the Gaussian process $\mathbb G$.
	
	\bigskip
	
	\emph{Details for Step 2.:}    Recall that by definition
	\begin{eqnarray*}
		\widehat \Sigma^*_n &= &  \frac{1}{n} \sum_{i=1}^n (\tilde{X}^*_i  - \bar{X}^*_n) (\tilde{X}^*_i  - \bar{X}^*_n)^\top 
		\\
		& = &  \frac{1}{n} \sum_{i=1}^n \tilde{X}^*_i (\tilde{X}^*_i)^\top   -  \bar{X}^*_n (\bar{X}^*_n)^\top  \\
		& = &  I_n + II_n
	\end{eqnarray*} 
	where $\tilde{X}^*_i = \widehat R^*_i  W_i$, with $\widehat R^*_i$ is randomly drawn from $\widehat{R}_1, \ldots, \widehat{R}_n$ given in (\ref{hatR})  and $W_1, \ldots, W_n$ are i.i.d. $ \sim \mathcal U$ and independent of $\widehat R_1, \ldots, \widehat R_n$.  We can write 
	\begin{eqnarray*}
		I_n =  \frac{1}{n} \sum_{i=1}^n (\widehat R^*_i)^2 W_i W^\top _i.
	\end{eqnarray*}
	Note that selecting $\widehat R^*_i$ from the sample $(\widehat R_1, \ldots, \widehat R_n)$ is equivalent to selecting $X^*_i$ from $(X_1, \ldots, X_n)$ and putting $\widehat R^*_i  = [(X^*_i - \bar{X}_n)^\top  \widehat \Sigma^{-1}_n  (X^*_i - \bar{X}_n)]^{1/2}$.  Thus,
	\begin{eqnarray*}
		I_n &=&  \frac{1}{n} \sum_{i=1}^n (X^*_i - \bar X_n)^\top  \widehat \Sigma^{-1}_n  (X^*_i - \bar{X}_n) W_i W^\top _i \\
		%&=& \frac{1}{n} \sum_{i=1}^n (X^*_i - \mu_0)^\top  \widehat \Sigma^{-1}_n  (X^*_i - \mu_0) W_i W^\top _i  \\
		%&& \  \ +   \ \frac{1}{n}(\mu_0 -  \bar{X}^*_n)^\top   \widehat \Sigma^{-1}_n  \sum_{i=1}^n (X^*_i - \bar{X}^*_n) W_i W^\top _i \\
		%&& \  \  +  \frac{1}{n}  \sum_{i=1}^n (X^*_i - \bar{X}^*_n)^\top  \widehat \Sigma^{-1}_n (\mu_0-\bar{X}^*_n)  W_i W^\top _i \\
		%& & \ \  +\  (\mu_0 -  \bar{X}^*_n)^\top   \widehat \Sigma^{-1}_n (\mu_0-\bar{X}^*_n) \frac{1}{n} \sum_{i=1}^n W_i W^\top _i 
		%& = & \frac{1}{n} \sum_{i=1}^n (X^*_i - \bar{X}_n)^\top  \widehat \Sigma^{-1}_n  (X^*_i - \bar{X}_n) W_i W^\top _i  \\
		%&& \  \ +   \ (\bar{X}_n -  \bar{X}^*_n)^\top   \widehat \Sigma^{-1}_n \frac{1}{n} \sum_{i=1}^n (X^*_i - \bar{X}^*_n) W_i W^\top _i \\
		%&& \  \  +  \frac{1}{n}  \sum_{i=1}^n (X^*_i - \bar{X}^*_n)^\top  \widehat \Sigma^{-1}_n (\bar{X}_n-\bar{X}^*_n)  W_i W^\top _i \\
		%& & \ \  +\  (\bar{X}_n -  \bar{X}^*_n)^\top   \widehat \Sigma^{-1}_n (\bar{X}_n-\bar{X}^*_n) \frac{1}{n} \sum_{i=1}^n W_i W^\top _i  \\
		%&& = I_{n1} + I_{n2}  + I_{n3} + I_{n4}. 
		%\end{eqnarray*}
		%Now, note that
		%\begin{eqnarray*}
		&= & \frac{1}{n} \sum_{i=1}^n (X^*_i - \bar{X}_n)^\top  \Sigma^{-1}_0  (X^*_i - \bar{X}_n) W_i W^\top _i \\
		& &  +
		\frac{1}{n} \sum_{i=1}^n (X^*_i - \bar{X}_n)^\top  (\widehat \Sigma^{-1}_n-\Sigma^{-1}_0)  (X^*_i - \bar{X}_n) W_i W^\top _i.
	\end{eqnarray*}
	Since $(X^*_1, \ldots, X^*_n)$ is independent of $(W_1, \ldots, W_n)$, it can be shown using the same techniques based on Poissonization described in the proof of \citet[Theorem 3.6.3]{vdvwell}  that
	\begin{eqnarray*}
		\frac{1}{n} \sum_{i=1}^n (X^*_i - \bar{X}_n)^\top  \Sigma^{-1}_0  (X^*_i - \bar{X}_n) W_i W^\top _i   
	\end{eqnarray*}
	has the same limit distribution as 
	\begin{eqnarray*}
		\frac{1}{n} \sum_{i=1}^n (X_i - \mu_0)^\top  \Sigma^{-1}_0  (X_i - \mu_0) W_i W^\top _i  =  \frac{1}{n} \sum_{i=1}^n R_i^2 W_i W_i^\top  =  \frac{1}{n} \sum_{i=1}^n Y_i Y_i^\top , 
	\end{eqnarray*}
	which is 
	\begin{eqnarray*}
		\sqrt n \left(\frac{1}{n} \sum_{i=1}^n Y_i Y_i^\top    -   \mathcal{I}_d\right)  \Rightarrow  \mathbb K_0.
	\end{eqnarray*}
	%Also,  using that  $\widehat \Sigma^{-1}_n-\Sigma^{-1}_0  = O_p(n^{-1/2})$ we can conclude that
	%\begin{eqnarray*}
	%\frac{1}{n} \sum_{i=1}^n (X^*_i - \bar{X}_n)^\top  (\widehat \Sigma^{-1}_n-\Sigma^{-1}_0)  (X^*_i - \bar{X}_n) W_i W^\top _i   = O_p(n^{-1/2}).
	%\end{eqnarray*}
	If $\mathbb P^*_n$ denotes the empirical process based on $X^*_1, \ldots, X^*_n$, then it is known conditionally on $X_1, X_2, \ldots,$ the process $\sqrt n (\mathbb P^*_n - \mathbb P_n)$ has the same limit distribution as $\sqrt n  (\mathbb P_n- \mathbb P)$; see \citet[Chapter 3.6]{vdvwell}.  This in turn implies that $\frac{1}{n} \sum_{i=1}^n \Vert X^*_i - \bar{X}_n) \Vert^2  =  \mathcal O_\PP(1)$. Also,
	\begin{eqnarray*}
		& & \left \Vert \frac{1}{n} \sum_{i=1}^n (X^*_i - \bar{X}_n)^\top  (\widehat \Sigma^{-1}_n-\Sigma^{-1}_0)  (X^*_i - \bar{X}_n) W_i W^\top _i \right  \Vert \\
		& \le  & \Vert \widehat \Sigma^{-1}_n-\Sigma^{-1}_0 \Vert \ \frac{1}{n} \sum_{i=1}^n \Vert X^*_i - \bar{X}_n \Vert^2   
		=    \mathcal O_\PP(n^{-1/2}),
	\end{eqnarray*}
	where we have used the fact that $\Vert v v^\top  \Vert =  \Vert v \Vert^2 $ and $\Vert W_i \Vert =1 $ since $W_i$ belongs to the unit sphere.   This allows us to conclude that
	\begin{eqnarray*}
		\sqrt n (I_n - \mathcal{I}_d) \Rightarrow   \mathbb K_0
	\end{eqnarray*}
	as $n \to \infty$.  Finally, we have that 
	\begin{align*}
	\bar{X}^*_n  = &\frac{1}{n} \sum_{i=1}^n \tilde{X}^*_i  \\
	= &  \frac{1}{n}  \sum_{i=1}^n  \left[(X^*_i  -  \bar{X}_n)^\top   \widehat \Sigma^{-1}_n   (X^*_i  -  \bar{X}_n)\right]^{1/2}  W_i \\
	= &  \frac{1}{n}  \sum_{i=1}^n  \Big[(X^*_i  -  \bar{X}_n)^\top   \left(\widehat \Sigma^{-1}_n  - \Sigma^{-1}_0\right) (X^*_i  -  \bar{X}_n) \\
	& \qquad \qquad   + (X^*_i  -  \bar{X}_n)^\top  \Sigma^{-1}_0 (X^*_i  -  \bar{X}_n) \Big]^{1/2}  W_i.
	\end{align*}
	Using the inequality $\vert (a + h)^{1/2}  - a^{1/2} \vert \le \sqrt 2  a^{-1/2} \vert h \vert $  for $a > 0$ and $h$ small enough so that $a+h \ge 0$ together with the triangle inequality we have that
	\begin{align*}
	& \left \Vert  \bar{X}^*_n - \frac{1}{n} \sum_{i=1}^n \left[(X^*_i - \bar X_n) \Sigma^{-1}_0  (X^*_i - \bar{X}_n)\right]^{1/2}  W_i  \right \Vert \\
	& \le   \ \frac{\sqrt 2}{n} \sum_{i=1}^n \frac{ \left \vert (X^*_i - \bar X_n) \left(\widehat \Sigma^{-1} _n - \Sigma^{-1}_0\right)  (X^*_i - \bar{X}_n)  \right \vert }{\left[ X^*_i - \bar X_n) \Sigma^{-1}_0  (X^*_i - \bar{X}_n)  \right]^{1/2}} \ \tag{since $\Vert W_i \Vert =1$} \\
	& \le \frac{\sqrt 2}{\lambda_m}  \  \Vert \widehat \Sigma^{-1} _n - \Sigma^{-1}_0  \Vert
	\end{align*} 
	where $\lambda_m > 0$ is the largest eigenvalue of $\Sigma_0$. Using again the fact that $\widehat \Sigma^{-1} _n - \Sigma^{-1}_0  = O_p(n^{-1/2})$ and that 
	$$
	\frac{1}{n} \sum_{i=1}^n \left[(X^*_i - \bar X_n) \Sigma^{-1}_0  (X^*_i - \bar{X}_n)\right]^{1/2}  W_i 
	$$ 
	has the same limit distribution as 
	$$
	\frac{1}{n} \sum_{i=1}^n \left[(X_i - \mu_0) \Sigma^{-1}_0  (X_i - \mu_0)\right]^{1/2}  W_i  = \frac{1}{n} \sum_{i=1}^n R_i W_i  =   \frac{1}{n} \sum_{i=1}^n Y_i 
	$$
	conditionally on $X_1, X_2, \ldots$  it follows that 
	\begin{eqnarray*}
		\sqrt n \bar{X}^*_n  \Rightarrow  \mathcal{N}(0, \mathcal{I}_d)
	\end{eqnarray*}
	conditionally on $X_1, X_2, \ldots$.  Thus, $\sqrt n II_n =  o_\PP(n^{-1/2})$ and we get finally \eqref{LimitSigmastar}. 
	
	\bigskip 
	
	\emph{Details for Step 3.} \quad  First, note that 
	\begin{eqnarray*}
		\Sigma^{-1/2}_0 \mathbb D  \stackrel{d}{=} \mathbb D_0 \sim \mathcal{N}(0, \mathcal{I}_d).
	\end{eqnarray*}
	Now, we will show that 
	\begin{eqnarray}\label{rel}  
	2 \mathbb S_0  \stackrel{d}{=}   \mathbb S \Sigma^{1/2}_0  +  \Sigma^{1/2}_0  \mathbb S.
	\end{eqnarray}
	To  this aim, recall that conditionally on $X_1, X_2, \ldots  $ the covariance matrix $\widehat \Sigma^*_n$ is asymptotically equivalent to 
	\begin{eqnarray*}
		\frac{1}{n} \sum_{i=1}^n Y_i Y^\top _i = \Sigma^{-1/2}_0  \frac{1}{n} \sum_{i=1}^n (X_i - \mu_0)^\top  (X_i-\mu_0)^\top   \Sigma^{-1/2}_0
	\end{eqnarray*}
	and the right side is itself equivalent to  $\Sigma^{-1/2}_0  \widehat \Sigma_n \Sigma^{-1/2}_0$.   Now, note that 
	\begin{align*}
	& \sqrt n ((\widehat \Sigma^*_n)^{-1/2}  - \mathcal{I}_d)\\
	&  =  \sqrt n ((\widehat \Sigma^*_n)^{-1}(\widehat{ \Sigma}^*_n)^{1/2}  - \mathcal{I}_d) \\
	& =  \sqrt n \left((\widehat \Sigma^*_n)^{-1}  - \mathcal{I}_d\right) (\widehat{ \Sigma}^*_n)^{1/2} +  \sqrt n \left( (\widehat{ \Sigma}^*_n)^{1/2}- \mathcal{I}_d \right)  \\
	& =   \sqrt n \left((\widehat \Sigma^*_n)^{-1}  - \mathcal{I}_d\right) (\widehat{ \Sigma}^*_n)^{1/2}- (\widehat{ \Sigma}^*_n)^{1/2} \sqrt n \left( (\widehat \Sigma^*_n)^{-1} -   \mathcal{I}_d \right)
	\end{align*}
	and therefore,
	\begin{eqnarray*}
		(\mathcal{I}_ d+  (\widehat{ \Sigma}^*_n)^{1/2}) \sqrt n ((\widehat \Sigma^*_n)^{-1/2}  - \mathcal{I}_d)    =  \sqrt n \left((\widehat \Sigma^*_n)^{-1}  - \mathcal{I}_d\right) (\widehat{ \Sigma}^*_n)^{1/2}.
	\end{eqnarray*}
	Now, the left side weakly converges to $2 \mathbb S_0$. As for the right side, we have $(\widehat{ \Sigma}^*_n)^{1/2}  \to \mathcal{I}_d$ in probability. Hence,  $ \sqrt n \left((\widehat \Sigma^*_n)^{-1}  - \mathcal{I}_d\right) (\widehat{ \Sigma}^*_n)^{1/2}$ has the same weak limit as   $\sqrt n \left((\widehat \Sigma^*_n)^{-1}  - \mathcal{I}_d\right)$ which in turn has the same  weak limit as 
	\begin{eqnarray*}
		&&\sqrt n \Sigma^{1/2}_0  \left(\widehat \Sigma^{-1}_n  - \Sigma^{-1}_0  \right)  \Sigma^{1/2}_0 \\
		&& =  \sqrt n \Sigma^{1/2}_0  \left(\widehat \Sigma^{-1/2}_n  \widehat \Sigma^{-1/2}_n- \Sigma^{-1/2}_0 \Sigma^{-1/2}_0  \right)  \Sigma^{1/2}_0\\
		&& = \Sigma^{1/2}_0  \left(  \sqrt n \left(\widehat \Sigma^{-1/2}_n  - \Sigma^{-1/2}_0 \right) \widehat \Sigma^{-1/2}_n  + \Sigma^{-1/2}_0  \sqrt n \left(\widehat \Sigma^{-1/2}_n  - \Sigma^{-1/2}_0 \right)\right)\Sigma^{1/2}_0 \\
		&& \Rightarrow  \Sigma^{1/2}_0  \mathbb S  +  \mathbb S \Sigma^{1/2}_0,
	\end{eqnarray*}
	from which we conclude the claimed result in \eqref{rel}.

\end{proof}

\bibliographystyle{ims}
\bibliography{sphericsym}

	\end{document}